\documentclass[11pt,a4paper]{amsart}

\usepackage[utf8]{inputenc}
\usepackage[T1]{fontenc}
\usepackage{lmodern}
\usepackage[protrusion=false]{microtype}

\usepackage{amsmath,amssymb,amsthm,mathtools}
\usepackage{color}
\usepackage[all]{xy}

\usepackage{geometry}
\newgeometry{inner=2.5cm, outer=2.5cm, top=3cm, bottom=4cm, headsep=1cm}

\usepackage{latexsym}
\usepackage{verbatim}

\usepackage{rotating}
\usepackage{pdflscape}

\usepackage{tikz}
\usetikzlibrary{cd, matrix, arrows, decorations.markings}
\usepackage{float}       

\usepackage{enumitem}
\setenumerate{listparindent=\parindent}
\setlist[enumerate]{
	font=\normalfont,
	label=(\roman*),
	topsep=3pt,
	itemsep=-0.3ex,
	partopsep=1em,
	parsep=1ex}

\usepackage[pdftex, draft=false]{hyperref} 
\hypersetup{
	pdfauthor = {Kevin Aguyar Brix, Alexander Mundey, and Adam Rennie},
	pdftitle = {Morphisms of Cuntz-Pimsner algebras from completely positive maps},
	linktocpage = true,
	colorlinks = true,
	linkcolor = blue,
	citecolor = blue 
}
\usepackage[noabbrev, capitalise, nameinlink]{cleveref}

\newcommand{\bma}{\left(\begin{array}{cc}}
\newcommand{\ema}{\end{array}\right)}
\newcommand{\bca}{\left(\begin{array}{c}}
\newcommand{\eca}{\end{array}\right)}

\newcommand{\cO}{\ensuremath{\mathcal{O}}}

\newcommand{\cT}{\ensuremath{\mathcal{T}}}

\newcommand{\bC}{\ensuremath{\mathbb{C}}}

\newcommand{\bT}{\ensuremath{\mathbb{T}}}

\newcommand{\la}{\left\langle}
\newcommand{\ra}{\right\rangle}

\numberwithin{equation}{section} 

\newtheorem{thm}{Theorem}[section]
\newtheorem{lemma}[thm]{Lemma}
\newtheorem{prop}[thm]{Proposition}
\newtheorem{corl}[thm]{Corollary}

\theoremstyle{definition} 
\newtheorem{defn}[thm]{Definition}
\newtheorem{construction}[thm]{Construction}
\newtheorem{ntn}[thm]{Notation}

\usepackage{etoolbox} 
\theoremstyle{remark} 
\newtheorem{rmk}[thm]{Remark}
\newtheorem{example}[thm]{Example}
\AtEndEnvironment{example}{\null\hfill$\bigtriangleup$}

\DeclareMathOperator{\End}{End}   
\DeclareMathOperator{\Aut}{Aut}  
\DeclareMathOperator{\Fock}{Fock} 
\DeclareMathOperator{\Id}{Id}     
\DeclareMathOperator{\linspan}{span} 
\DeclareMathOperator{\Mult}{Mult} 
\DeclareMathOperator{\KSGNS}{KSGNS}

\newcommand{\catname}[1]{{\normalfont \mathsf{#1}}}
\newcommand{\Ench}{\catname{Enchilada}}
\newcommand{\Quesa}{\catname{Quesadilla}}

\newcommand{\rInd}{\operatorname{\mbox{$r$-$\operatorname{Ind}$}}}

\newcommand{\C}{\mathbb{C}}   
\newcommand{\wt}{\widetilde}

\newcommand{\N}{\mathbb{N}}   
\newcommand{\ol}{\overline}  
\newcommand{\ox}{\otimes}     

\newcommand{\T}{\mathcal{T}} 
\newcommand{\Z}{\mathbb{Z}}   
\newcommand{\stroke}{\mathbin|}     

\newcommand{\cor}[4]{(#1,{}_{#2} #3_{#4})} 

\def\pairL_#1(#2|#3){{}_{#1}(#2\stroke#3)} 
\def\pairR(#1|#2)_#3{(#1\stroke#2)_{#3}} 
\def\scal<#1|#2>{\langle#1\stroke#2\rangle} 

\makeatletter
\newcommand{\tpmod}[1]{{\@displayfalse\pmod{#1}}}
\makeatother

\title{\vspace*{-1pc}%
	Morphisms of Cuntz--Pimsner algebras from completely positive maps
}
\author[K. A. Brix]{Kevin Aguyar Brix}
\email{kabrix.math@fastmail.com}
\address[K. A. Brix]{School of Mathematics and Statistics, University of Glasgow, United Kingdom}
\author[A. Mundey]{Alexander Mundey}
\email{amundey@uow.edu.au}
\author[A. Rennie]{Adam Rennie}
\email{renniea@uow.edu.au}
\address[A. Mundey and A. Rennie]{School of Mathematics and Applied Statistics, University of Wollongong, Australia}

\thanks{The authors thank Kylie Fairhall for tex-mex advice, and Menevse Ery\"{u}zl\"{u} for sharing her constructions \cite{E21}. 
	K.A. Brix was supported by a Carlsberg Foundation Internationalisation Fellowship and a DFF-International postdoc (Case number 1025--00004B).
	A. Mundey was supported by University of Wollongong AEGiS CONNECT grant 141765.}
\date{\today}
\dedicatory{In memory of Iain Raeburn,\\ with gratitude for all of his contributions to mathematics and our community.}

\makeatletter
\@namedef{subjclassname@2020}{%
  \textup{2020} Mathematics Subject Classification}
\makeatother


\begin{document}

	\begin{abstract}
		We introduce positive correspondences as right $C^*$-modules with left actions given by completely positive maps. 
		Positive correspondences form a semi-category that contains the \mbox{$C^*$-correspondence} (Enchilada) category as a ``retract''. 
		Kasparov's KSGNS construction provides a semi-functor from this semi-category onto the \mbox{$C^*$-correspondence} category.
		The need for left actions by completely positive maps appears naturally when we consider morphisms between Cuntz--Pimsner algebras,
		and we describe classes of examples arising from projections on \mbox{$C^*$-correspondences} and Fock spaces, as well as examples from conjugation by bi-Hilbertian bimodules of finite index.
	\end{abstract}

\maketitle
\vspace{-2pc}


\parskip=3pt
\parindent=0pt

\addtocontents{toc}{\vspace{-1pc}}

\section{Introduction}
\label{sec:intro}

In this paper we introduce positive correspondences between \mbox{$C^*$-algebras} and examine their interplay with Cuntz--Pimsner algebras. 
In particular, we look at positive correspondences arising from both correspondence projections and conjugations by bi-Hilbertian bimodules, 
and use these to construct \mbox{$C^*$-correspondences} between Cuntz--Pimsner algebras. 
In the general framework of noncommutative geometry, \mbox{$C^*$-algebras} appear as noncommutative topological spaces, 
making them highly suited to study topological dynamical systems (e.g.  symbolic dynamics, directed graphs, $k$-graphs) as well as more general noncommutative dynamics.

For good and bad, $*$-homomorphisms between \mbox{$C^*$-algebras} are very rigid, which can make them hard to find in practice. 
When a right $C^*$-module (a.k.a. a Hilbert module) $X_B$ over a \mbox{$C^*$-algebra} $B$ admits an adjointable left action of a possibly different \mbox{$C^*$-algebra} $A$, we call ${}_AX_B$ a $C^*$-correspondence; 
and we may think of it as a more flexible notion of morphism from $A$ to $B$. 
Indeed, every $*$-homomorphism induces a \mbox{$C^*$-correspondence}.
This perspective of \mbox{$C^*$-correspondences} as morphisms is taken in \cite{Meyer-Sehnem,EKQR,EKQR-littletaco}. 
Accordingly, \mbox{$C^*$-correspondences} provide a flexible framework comparing the structure of \mbox{$C^*$-algebras}, 
and allow us to consider more general classes of dynamics which are not induced by $*$-homomorphisms.

Pimsner \cite{Pimsner} pioneered a construction of a \mbox{$C^*$-algebra} built from a self-correspondence over a (possibly noncommutive) \mbox{$C^*$-algebra.}  
Pimsner's \mbox{$C^*$-algebras} are universal, and generalise both crossed products by $\Z$ and Cuntz--Krieger algebras \cite{Cuntz-Krieger1980}. 
This construction was later refined by Katsura \cite{KatsuraCPalg}, and we refer to these \mbox{$C^*$-algebras} as \emph{Cuntz--Pimsner algebras}. 
In a sense, these algebras encode the dynamics of a \mbox{$C^*$-correspondence}, and they have found considerable use in analysing topological dynamical systems 
\cite{KatsuraTopGraph, Kajiwara-Watatani2005, Nekrashevych2009, LRRW, Kakariadis-Katsoulis, Carlsen-Doron-Eilers}.

Typically, it is difficult to find maps between Cuntz--Pimnser algebras which are induced by maps between their defining \mbox{$C^*$-correspondences}. 
Such maps are important, since for commutative algebras they can be considered ``dynamically compatible'' maps. 
The coisometric morphisms of \cite[Definition~1.3]{BrenkenTopQuiv} induce $*$-homomorphisms between Cuntz--Pimsner algebras, 
but they are often hard to come by. More generally, the covariant correspondences considered in \cite{Meyer-Sehnem,E21} induce \mbox{$C^*$-correspondences} between Cuntz--Pimsner algebras, 
but again these are somewhat rare. 

In our previous work \cite{Grooves}---inspired by moves on graphs \cite{Williams1973, Bates-Pask2004, ERRS2016,MPT}---we developed notions of splittings for \mbox{$C^*$-correspondences} 
and showed how they preserve the isomorphism class or Morita equivalence class of the Cuntz--Pimsner algebras. 
Here, we take a different approach.

Our starting point is the observation that for a self-correspondence ${}_AX_A$, and a sub-self-correspondence ${}_AY_A\subseteq {}_AX_A$,
the inclusion/projection relation does {\em not}, in general, give rise to a covariant morphism of correspondences in either direction.
We do obtain a $*$-homomorphism from the Toeplitz algebra of ${}_A Y_A$ to the Toeplitz algebra of ${}_A X_A$, 
but this $*$-homomorphism does not usually pass to the Cuntz--Pimsner quotients.
Instead, we obtain a completely positive map from the Cuntz--Pimsner algebra of ${}_A X_A$ to the Cuntz--Pimsner algebra of ${}_A Y_A$. 
The KSGNS construction \cite[Chapter 5]{Lance} allows us to pass to an honest $C^*$-correspondence.

A categorical framework was developed in \cite{EKQR-littletaco,EKQR}, in which morphisms are isomorphism classes of \mbox{$C^*$-correspondences}. 
This correspondence category (a.k.a. the Enchilada category) extends the more traditional category of \mbox{$C^*$-algebras} with $*$-homomorphisms. 
The correspondence category was further explored in \cite{Meyer-Sehnem,EKQ, E21}.

To encapsulate both completely positive maps and \mbox{$C^*$-correspondences} in a single object we introduce \emph{positive correspondences} as right $C^*$-modules admitting a left action by a (strict) completely positive map (see \cref{def:pozzy-cozzy}). A similar setup can be found in \cite[Ch. 5]{Lance}.
The isomorphism classes of positive correspondences form a semi-category $\Quesa$ (see \cref{def:semi-category}), and this seems to be the right domain for the KSGNS construction. 
In fact, the semi-category contains $\Ench$ as a wide subcategory, and KSGNS acts as a semi-functor from our semi-category onto $\Ench$, which now appears as a retract of $\Quesa$. 
The only reason why positive correspondences do not form an honest category is a lack of identities:
the obvious identity morphisms act as identities on the right, but not on the left (see \cref{lem:all-fucked-up}).

There is an abundance of examples giving rise to completely positive maps and thereby positive correspondences, and we outline a few in this paper.
We note that this work is different, but related, to the crossed products by completely positive maps considered in \cite{KwasniewskiCP,BrenkenCP}. 

After some general background in \cref{sec:oncemoreforthecheapseats}, we recall the KSGNS construction in detail in \cref{sub:quesadillas}, 
and show how we obtain a semi-category from positive correspondences.
In \cref{sec:projections,sec:bi-curious}, we show how a range of different constructions on \mbox{$C^*$-correspondences} and their Fock modules give rise to positive correspondences. The three types of positive correspondences we consider come from projections
\begin{enumerate}
	\item on self-correspondences themselves,
	\item on Fock modules of self-correspondences, and
	\item on bi-Hilbertian bimodule ``amplifications'' of the Fock modules of self-correspondences.
\end{enumerate}
In each case  we describe the resulting positive maps, and the induced \mbox{$C^*$-correspondences} we obtain from the KSGNS construction. For the second class of examples we also describe their relation to subproduct systems. A feature of the third class of examples is that we can consider pairs of Cuntz--Pimsner algebras associated to self-correspondences with differing coefficient algebras. 


\section{Correspondences, Cuntz--Pimsner algebras, and Enchiladas}
\label{sec:oncemoreforthecheapseats}
In this section we collect definitions and basic results, while establishing notation.
\subsection{\texorpdfstring{$C^*$}{C*}-modules and correspondences}
\label{subsec:mods-coz}

We refer the reader to \cite{Lance,RaeburnWilliams} for background on $C^*$-modules. 
Throughout this paper, $A$ and $B$ denote separable \mbox{$C^*$-algebras}. 
Given a right $C^*$-module over $A$ (a.k.a. a right Hilbert $A$-module or right $A$-module) $X$---written $X_A$ when we want to remember the coefficient
algebra---we denote the \mbox{$C^*$-algebra} of adjointable operators on $X$ by $\End_A(X)$, the $C^*$-ideal of compact operators by $\End^0_A(X)$, and the finite-rank operators by
$\End_A^{00}(X)$. The finite-rank operators are generated
by rank-one operators $\Theta_{x,y}$ with $x,\,y \in X$ satisfying $\Theta_{x,y}(z) = x \cdot\la y \mid z\ra_A$.
If $A_A$ is the trivial $A$-module, then we identify $\End_A(X)$ with the multiplier algebra $\Mult(A)$ of $A$.

A right $C^*$-module $X_A$ is {\em full} if $A = \ol{\linspan}\langle X \mid X \rangle_A$.
We also have occasion to consider left $A$-modules, which behave analogously.

\begin{defn}
  \label{defn:correspondence}
  Let $X_B$ be a countably generated right $C^*$-module, and let $\phi\colon A \to\End_B(X)$ be a $*$-homomorphism. 
  The data $(\phi, {}_AX_B)$ is a \emph{$C^*$-correspondence} (or an $A$--$B$-\emph{correspondence}). 
  If $\phi$ is understood we just write ${}_AX_B$.
  If $\overline{\phi(A)X}=X$, we say ${}_AX_B$ is \emph{nondegenerate}. 
  If $A=B$ we say that $(\phi,{}_AX_A)$ is a {\em \mbox{$C^*$-correspondence} over $A$} or a \emph{self-correspondence}. 
  If $\phi$ is understood we sometimes use the notation $a \cdot x \coloneqq \phi(a)x$ for $a \in A$ and $x \in X_B$. 
  A \mbox{$C^*$-correspondence} is \emph{regular} if $\phi$ is injective and $\phi(A)\subseteq \End^0_B(X)$.
\end{defn}

\begin{defn}
 A regular correspondence $(\phi, {}_AX_B)$ with $X_B$ full and $\End^0_B(X)=\phi(A)$ is called a \emph{Morita equivalence bimodule. }
Morita equivalence bimodules are also left $A$-modules with inner product ${}_A\langle x,y\rangle=\phi^{-1}(\Theta_{x,y})$.
\end{defn}

 We discuss the more general class of bi-Hilbertian bimodules in \cref{sec:bi-curious}.

An important technical and computational tool in the theory and applications
of $C^*$-modules are frames. These are as close as we can get to an orthonormal basis in a $C^*$-module, and serve similar purposes.
\begin{defn}
	\label{defn:frame}
	A (countable) \emph{frame}\footnote{In the signal analysis literature our `frame' would be a normalised tight frame.} for a
	right $A$-module $X$ is a sequence $(u_j)_{j\geq 1}\subset
	X$ such that
	$
	x = \sum_{j \ge 1} u_j \langle u_j \mid x \rangle_A
	$
	for all $x \in X$, with convergence in norm.
\end{defn}
If $(u_j)_{j \ge 1}$ is a frame for $X$, then $X$ is generated as a right $A$-module by
the $u_j$, so $X$ is countably generated. Conversely, every countably generated right $A$-module admits a frame \cite[Corollary~3.3]{RaeburnThompson}.

We may compose \mbox{$C^*$-correspondences} using the internal tensor product. Given correspondences 
$(\phi,{}_AX_B)$ and $(\psi,{}_BY_C)$, we obtain a new correspondence $(\phi\ox{\rm Id}_Y,{}_A(X\ox_BY)_C)$, where $X\ox_B Y$ is the quotient of $X\ox_\C Y$ by the closed submodule generated by
\[
\{xb\ox y-x\ox\psi(b)y:x\in X,\ y\in Y,\ b\in B\},
\] 
\cite[Proposition 4.5]{Lance}. 
It is a non-trivial result that $X\ox_BY$ is a right $C^*$-module with inner product $\langle x_1\ox y_1\mid x_2\ox y_2\rangle_C\coloneqq\langle y_1\mid \psi(\langle x_1\mid x_2\rangle_B)y_2\rangle_C$,
for all $x_1,x_2\in X$ and $y_1,y_2\in Y$.

\begin{defn}[{cf.~\cite[Definition~1.1]{BrenkenTopQuiv}}]
  A \emph{correspondence morphism} 
  $(\pi,\psi)$ from $(\phi_X, {}_AX_A)$ to $(\phi_Y,  {}_BY_B)$ consists of a $*$-homomorphism $\pi \colon A \to B$ together with a linear map $\psi \colon X \to Y$ satisfying:
  \begin{enumerate}[label=(\roman*), ref=(\roman*)]
      \item $\pi(\la\xi \mid \eta\ra_A) = \la\psi(\xi) \mid \psi(\eta)\ra_B$ for all $\xi,\, \eta \in X$;
      \item $\phi_Y(\pi(a)) \psi(\xi) = \psi(\phi_X(a) \xi)$ for all $a \in A$ and $\xi \in X$; and
      \item $\psi(\xi) \cdot \pi(a) = \psi(\xi \cdot a)$ for all $a \in A$ and $\xi \in X$.
  \end{enumerate}
  We write $(\pi,\psi) \colon (\phi_X, {}_AX_A) \to (\phi_Y,  {}_BY_B)$.
  If $(\phi_Y,{}_BY_B) = (\Id,{}_BB_B)$, then we call $(\pi,\psi)$ a \emph{representation} of  $(\phi_X, {}_AX_A)$ in $B$.
\end{defn}

Let $(\phi,{}_AX_A)$ be a \mbox{$C^*$-correspondence} and let
$\ker(\phi)^\perp \coloneqq\{a\in A:ab=0\ \mbox{for all }b\in\ker(\phi)\}$.
Following Katsura \cite{KatsuraCPalg}, the \emph{covariance ideal} of $(\phi,{}_AX_A)$ is given by
\begin{equation}
  J_X \coloneqq \phi^{-1}(\End_A^0(X)) \cap \ker(\phi)^{\perp} \triangleleft A.
\label{eq:cov-ideal}
\end{equation}
 The covariance ideal is the largest ideal of $A$ on which $\phi$ is injective with image contained in $\End_A^0(X)$. 

We set the scene for a restricted class of {\em covariant morphisms} which respect the covariance ideal.
The following well-known characterisation of the compact operators on $X_A$ will be helpful in formulating covariance.  
\begin{ntn}
	\label{ntn:compacts}
	If $X_A$ is a right $A$-module then we let $X^*$ denote the \emph{conjugate module} of $X$, which is a left $A$-module. 
    We identify $\End^0_A(X)$ with $X\ox_AX^*$ as $\End^0_A(X)$-modules via the map $\Theta_{x,y} \mapsto x \ox y^*$. 
    We also give $X\ox_AX^*$ the structure of a \mbox{$C^*$-algebra} inherited from $\End^0_A(X)$.
\end{ntn}

\begin{lemma}[{\cite[Lemma 2.2]{KajPinWat98}}]
	Let $(\pi,\psi) \colon (\phi_A,{}_A Y_A) \to (\phi_B, {}_B X_B)$ be a correspondence morphism. 
    There is an induced $*$-homomorphism $\psi^{(1)} \colon \End^0_A(Y)\to \End^0_B(X)$ satisfying $\psi^{(1)}(y_1 \ox y_2^*)=\psi(y_1) \ox\psi(y_2)^*$, for all $y_1,y_2\in Y$.
\end{lemma}

\begin{defn}[{cf.~\cite[Definition~1.3]{BrenkenTopQuiv}}]
  We say that $(\pi,\psi) \colon (\phi_A,{}_A Y_A) \to (\phi_B, {}_B X_B)$ is \emph{covariant} if for all $a \in J_Y$
  we have $\psi^{(1)} \circ \phi_Y(a) = \phi_X \circ \pi(a)$.
\end{defn}

Correspondence morphisms between \mbox{$C^*$-correspondences} induce $*$-homomorphisms between the associated Toeplitz algebras, 
and covariant morphisms descend to $*$-homomorphisms of the associated Cuntz--Pimsner algebras \cite[Proposition~1.4]{BrenkenTopQuiv}.
We now turn to a description of these algebras.

\subsection{Toeplitz--Pimsner and Cuntz--Pimsner algebras}
\label{subsec:TP-CP}

Fix a \mbox{$C^*$-correspondence} $(\phi,{}_AX_A)$. Define $X^{\otimes 0} \coloneqq {_A}A_A$, $X^{\otimes 1} \coloneqq X$, and $X^{\otimes n+1} \coloneqq
X^{\otimes n} \otimes_A X$ for $n \ge 1$. The \emph{Fock module} of $X$ is the
$\ell^2$-direct sum $\Fock_X \coloneqq \bigoplus_{n=0}^\infty X^{\otimes n}$ regarded as a
correspondence over $A$ with diagonal left action. If $x \in X^{\ox n}$ for some $n \ge 0$ we call $n$ the \emph{length} of $x$ and write $|x| = n$.

As in \cite{Pimsner}, the \emph{Toeplitz algebra} $\cT_X$ is the $C^*$-subalgebra of
$\End_A(\Fock_X)$ generated by the creation operators $T_x$, $x\in X$ which satisfy
$$
T_x(x_1\ox x_2\ox\cdots\ox x_k)\coloneqq x\ox x_1\ox x_2\ox\cdots\ox x_k
$$
on elementary tensors $x_1\ox x_2\ox\cdots\ox x_k \in \Fock_X$. 
The adjoint $T^*_x$ satisfies
\[
T^*_x(x_1 \ox \cdots \ox x_k) = \la x\mid x_1\ra_A\cdot x_2 \ox \cdots \ox x_k
\]
for
$k \ge 1$, and $T^*_x|_{X^{\otimes0}} = 0$. 
For $a \in A$ we let $T_a$ denote the operator of left multiplication by $a$, given on simple tensors by
$T_a(x_1\ox\cdots\ox x_k)=a\cdot x_1\ox\cdots\ox x_k$.

Following \cite{KatsuraCPalg}, with $J_X$ the covariance ideal of $X$, the algebra of compact operators $\End_A^0(\Fock_X\cdot J_X)$ is an ideal of the Toeplitz-Pimsner algebra $\cT_X$. 
The \emph{Cuntz--Pimsner algebra} $\cO_X$ is defined to be the quotient $\cT_X/\End_A^0(\Fock_X\cdot J_X)$. 
Thus, we have an exact sequence
\begin{equation}
0\to \End_A^0(\Fock_X\cdot J_X) \longrightarrow \cT_X \stackrel{q}{\longrightarrow} \cO_X\to 0.
\label{eq:starting-sequence}
\end{equation}

The maps $j_X \colon x \mapsto T_x$ and
$j_A \colon a \mapsto T_a$ constitute a representation of $(\phi,{}_A X_A)$ whose image generates $\cT_X$. This
representation is universal: for any 
representation $(\psi, \pi)\colon (\phi,{}_A X_A) \to B$, there is a unique $*$-homomorphism
$\psi \times \pi \colon\cT_X \to B$ such that $(\psi \times \pi) \circ j_X = \psi$ and
$(\psi \times \pi) \circ j_A = \psi$ (see
\cite[Theorem~3.4]{Pimsner}).

The Cuntz--Pimsner algebra $\cO_X$ is
generated by the covariant representation of $(\phi,{}_A X_A)$ given by
$i_X \colon  x\mapsto S_x\coloneqq q(T_x)$ and $i_A \colon  a\mapsto S_a\coloneqq q(T_a)$. This representation is universal: for every covariant representation $(\psi, \pi) \colon (\phi,{}_A X_A) \to B$ there is a unique $*$-homomorphism
$\psi \times \pi \colon\cO_X \to B$ such that $(\psi \times \pi) \circ i_X = \psi$ and
$(\psi \times \pi) \circ i_A = \psi$.

\subsection{\texorpdfstring{$A$}{A}-\texorpdfstring{$C^*$}{C*}-algebras}

For a locally compact Hausdorff space $X$, the notion of a $C_0(X)$-algebra has been used for many years to handle localisation and limits in homological contexts, e.g. \cite{KasNov}. 
Here, we relax and extend the notion by allowing a noncommutative algebra in place of $C_0(X)$ and we do not assume a central image.
\begin{defn}
 Fix a \mbox{$C^*$-algebra} $A$. 
 A \mbox{$C^*$-algebra} $C$ is an \emph{$A$-$C^*$-algebra} (or simply an \emph{$A$-algebra}) if there is a nondegenerate $*$-homomorphism $\phi \colon A\to \Mult(C)$. 
 If $C$ and $D$ are $A$-algebras with $\phi_C \colon A \to \Mult(C)$ and $\phi_D \colon A \to  \Mult(D)$, then an \emph{$A$-algebra morphism} is 
 a nondegenerate $*$-homomorphism $\alpha \colon C \to D$ such that $\phi_D = \widetilde{\alpha} \circ \phi_C$, 
 where $\widetilde{\alpha}$ is the extension of $\alpha$ to the corresponding multiplier algebras \cite[Proposition 2.1]{Lance}).
\end{defn}

When $A$ is commutative, our notion of an $A$-algebra is more relaxed than the established notion of a $C_0(X)$-algebra.
However, the advantage is that a Cuntz--Pimsner algebra over an $A$-module is an $A$-algebra in our sense, as we show in the next example.

\begin{example} 	
Suppose $(\phi,{}_{A}X_A)$ is a \mbox{$C^*$-correspondence} and let $j_A \colon A \to \cT_X$  be the universal inclusion, which we think of as the diagonal representation of $A$ on $\Fock_X$. 
Since $\End_{J_X}^0(\Fock_X \cdot J_X)$ is an ideal in $\cT_X$, the map $j_A$ induces the structure of an $A$-\mbox{$C^*$-algebra} on $\End_{J_X}^0(\Fock_X \cdot J_X)$. 
In particular, the defining exact sequence for Cuntz--Pimsner algebras,
\begin{equation}	
\begin{tikzcd}
    0 \arrow[r] & \End_{J_X}^0(\Fock_X \cdot J_X) \arrow[r] & \cT_X \arrow[r] & \cO_X \arrow[r] & 0
\end{tikzcd},
\label{eq:def-ext}
\end{equation}
may be interpreted as an exact sequence of $A$-\mbox{$C^*$-algebras}. 
\end{example}

A \mbox{$C^*$-algebra} $C$ being an $A$-algebra with $\phi \colon A \to \Mult(C)$ nondegenerate is equivalent to $(\phi,{}_A C_C)$ being a nondegenerate $C^*$-correspondence.
We record the following facts about the relationship between \mbox{$C^*$-correspondences} and $A$-algebras.
\begin{lemma}
  \label{lem:bimod-algs}
  Let $(\phi,{}_A F_B)$ be a nondegenerate \mbox{$C^*$-correspondence}, and let $C$ be a $B$-algebra. 
  Then the following are true:
  \begin{enumerate}[label=(\roman*), ref=(\roman*)]
      \item $F\ox_BC$ is an
            $\End^0_C(F \ox _B C)$--$C$-Morita equivalence  bimodule;
      \item $(F \ox_B C)^* \cong C {}_B\ox F^*$, where ${}_B\ox$ denotes the tensor product of left correspondences;
      \item $\End^0_C(F \ox _B C) \cong (F \ox_B C) \ox_C  (C {}_B\ox F^*) $ is an $A$-algebra; and
      \item $(F \ox_B C) \ox_C  (C {}_B\ox F^*)$ can be identified with $F\ox C\ox F^*$  via the map 
      \[
      (f_1 \ox c_1) \ox (c_2 \ox f_2^*) \mapsto f_1 \ox c_1 c_2 \ox f_2^*,
      \] 
\end{enumerate}		
with the adjoint satisfying $(f_1\ox c_1\ox f_2^*)^*=f_2\ox c_1^*\ox f_1^*$ and the product satisfying
  \[
      (f_1\ox c_1\ox f_2^*)(f_3\ox c_2\ox f_4^*)
      =f_1\ox c_1\langle f_2|f_3\rangle_Bc_2\ox f_4^*,
  \]
  for all $f_1,f_2,f_3,f_4\in F$ and $c_1,c_2\in C$.
\end{lemma}

\begin{lemma}\label{lem:compacts-faithful}
Let $F_B$ be a right $B$-module and let $C$ be a $B$-algebra.
Suppose $(\phi,{}_{C}X_D)$ is a nondegenerate \mbox{$C^*$-correspondence} with $\phi$ injective. 
Then the action of $\End_C^0(F \ox_B C)$ on $F \ox_B X_D \coloneqq F \ox_B C \ox X_D$, defined on rank-one operators by
  \[
      \Theta_{f_1 \ox c_1, f_2 \ox c_2} (f_3 \ox x) = f_1 \ox \phi(c_1 \langle f_2 \mid f_3\rangle_B c_2^*) x,
  \]
  for all $f_1,f_2,f_3\in F$ and $c_1,c_2\in C$, and $x\in X$, is adjointable and injective.
\end{lemma}

\begin{proof}
  The left action of $\End_C^0(F \ox_B C)$ on $F \ox_B X_D$ is induced by the map $T \mapsto T \ox \Id_X$ from $\End_C^0(F \ox_B C)$ to $\End_D (F \ox_B C \ox_{\phi} X_D)$, so the action is adjointable.
  Since $\End_C^0(F \ox_B C)$ acts faithfully on $F \ox_B C$ and $\phi$ is injective, \cite[Lemma 4.7]{KatsuraCPalg} implies that $T \mapsto T \ox \Id_X$ is injective.
\end{proof}

\begin{lemma}\label{lem:modulequotients}
Let $F_B$ be a right $B$-module, let $C$ be a $B$-\mbox{$C^*$-algebra,}  and let $I$ be an ideal in $C$. 
Then $I$ is a $B$-\mbox{$C^*$-algebra,}  $\End_I^{0} (F \ox_B I)$ is an ideal in $\End_C^0(F \ox_B C)$, and the quotient by this ideal is isomorphic to $\End_{C/I}^0 (F \ox_B (C/I))$.
\end{lemma}

\begin{proof}
Let $\phi \colon B \to \Mult(C)$ be a nondegenerate $*$-homomorphism and let $(e_i)_i$ be an approximate unit for $I$. 
For each $d \in I$, we have $\phi(b)d = \lim_i \phi(b)d e_i \in I$, so $\phi$ induces a nondegenerate $*$-homomorphism $\phi_I \colon B \to \Mult(I)$. 

Applying~\cite[Lemma 1.6]{KatsuraCP_Ideals}, we see that $\End_I^{0} (F \ox_B I)$ is an ideal of $\End_C^0(F \ox_B C)$ and that the quotient is isomorphic to $\End_{C/I}^0((F \ox_B C)/(F \ox_B I))$, 
where $(F \ox_B C)/(F \ox_B I)$ is the quotient vector space equipped with the structure of a right $C/I$-module.

The algebra $C/I$ inherits a $B$-algebra structure from $\phi \colon B \to \Mult(C)$.
It remains to show that $(F \ox_B C)/(F \ox_B I)$ is isomorphic to $F \ox_B (C/I)$ as right $C/I$-modules. 
A routine computation shows that the map $f \ox c + F \ox_B I \mapsto f \ox (c + I)$ 
extends to a well-defined isometric linear map from $(F \ox_B C)/(F \ox_B I)$ to $F \ox_B (C/I)$. 
Right $C/I$-linearity is routine to check.
\end{proof}

\section{Positive Correspondences and Quesadillas}
\label{sub:quesadillas}

The idea of using some form of correspondences as morphisms has been around for a long time and has appeared in numerous contexts. In the $C^*$-algebraic setting the correspondence category was formalised in \cite{EKQR}. We summarise the key points. 

\begin{defn}\label{defn:correspondence_category}
	The \emph{correspondence category} (or \emph{Enchilada category}) $\Ench$ is the category with objects given by \mbox{$C^*$-algebras} and such that the set of morphisms from a \mbox{$C^*$-algebra} $A$ to a \mbox{$C^*$-algebra} $B$ is the collection of isomorphism classes of nondegenerate $A$--$B$-correspondences. Composition is given by the isomorphism class of the balanced tensor product of any representatives. 
\end{defn}

\begin{rmk}
	In \cref{defn:correspondence_category} one needs to take isomorphism classes of \mbox{$C^*$-correspondences} to ensure associativity of composition. Without taking isomorphism classes one instead ends up with a bicategory (see \cite{Meyer-Sehnem}) as associativity is only defined up to isomorphism.
\end{rmk}

We will now set about extending $\Ench$ by including additional morphisms coming from completely positive maps. 

\subsection{The KSGNS construction}
\label{subsec:KSGNS}
In this section we recall the details of the KSGNS\footnote{Kasparov--Stinespring--Gelfand--Naimark--Segal} construction, and we refer to \cite{Kas-SV} and \cite[Chapter 5]{Lance} who we follow here.

Let $X_A$ and $Y_A$ be right $A$-modules and let $\End_A(X,Y)$ denote the adjointable $A$-linear operators from $X_A$ to $Y_A$. Recall that the \emph{strong-* topology} on $\End_A(X, Y)$ is generated by the seminorms
\[
T\mapsto || Tx ||_Y \quad \textrm{and} \quad T\mapsto || T^*y ||_X,
\]
for all $x\in X$ and $y\in Y$. 
The strong-* topology agrees with the \emph{strict topology} on norm-bounded sets~\cite[Proposition~8.1]{Lance}. 
We only consider these topologies on bounded sets and so only refer to the strict topology. 
A bounded net $(T_i)_{i \in \N}$ converges strictly to $T$ in $\End_A(X, Y)$ precisely when
\[
|| (T_i - T)x ||_Y \to 0 \quad \textrm{and} \quad || (T^*_i - T^*)y ||_X \to 0
\]
for every $x\in X$ and $y\in Y$.  
\begin{rmk}
	If $X_A$ is a right $A$-module, then $(u_j)_{j \ge 1}$ in $X_A$ is a frame if and only if 
	$\sum_{j\geq 1}\Theta_{u_j,u_j}$ converges strictly to ${\Id}_X$ in $\End_A(X)$. 
\end{rmk}

\begin{defn}
\label{defn:nice-pozzies}
If $A$ is a \mbox{$C^*$-algebra} and $X$ is a $B$-module, then a completely positive map $\rho \colon A \to \End_B(X)$ is \emph{strict}
if $(\rho(a_i))_i$ is strictly Cauchy for some approximate unit $(a_i)_i$ for $A$.
Equivalently, $\rho$ is strict if and only if there is a completely positive map $\overline{\rho}\colon \Mult(A) \to \End_B(X)$
which is strictly continuous on the unit ball and whose restriction to $A$ is $\rho$, (cf.~\cite[Corollary 5.7]{Lance}).
A strict map $\rho\colon A\to\End_B(X)$ is {\em nondegenerate} if $\rho(a_i)\to {\rm Id}_X$ strictly.
\end{defn}
\begin{lemma}
	If $\rho \colon A \to B$ and ${\sigma} \colon B \to C$ are strict completely positive maps between \mbox{\mbox{$C^*$-algebras}},
	then ${\sigma}\circ \rho \colon A \to C$ is a strict completely positive map.
\end{lemma}
\begin{proof}
	Compositions of completely positive maps are completely positive. 
	Since $\rho$ is strict, it lifts to a completely positive map $\ol{\rho} \colon \End_A(A) \to \End_B(B)$ and similarly for $\sigma$.
	Then ${\sigma}\circ \rho$ is strict since $\ol{{\sigma}} \circ \ol{\rho}$ is a lift.
\end{proof}

\begin{example}
If $\phi \colon A \to \End_B(X)$ is a $*$-homomorphism, then $\phi$ is strict if and only if $\ol{\phi(A)X_B}$ is a complemented submodule of $X_B$ \cite[Proposition~5.8]{Lance}. In particular, if $\phi \colon A \to B$ is a surjective $*$-homomorphism, then $\phi$ is strict. 
\end{example}
\begin{example}
If $B \subseteq A$ and $\rho\colon A \to B$ is contractive and idempotent, then $\rho$ is strict. 
Indeed, by Tomiyama's theorem~\cite[Theorem 1.5.10]{Brown-Ozawa}, $\rho$ is a conditional expectation,
so if $(a_i)$ is an approximate unit for $A$, then
\[
    \|\rho(a_i)b - \rho(a_j)b\| = \|\rho(a_ib - a_j b)\| \to 0,
\]
for $b\in B$, so $\rho$ is strict.
\end{example}

We now recall the KSGNS construction. 

\begin{construction}[KSGNS]
Let $A$ and $B$ be \mbox{$C^*$-algebras}, let $X_B$ be a $B$-module, and let $\rho \colon A \to \End_B(X)$ be a strict completely positive map.
The algebraic tensor product $A \odot X$ is a right $B$-module and a semi inner-product space with respect to the sesquilinear form defined by
\begin{equation}
\la a_1 \ox x_1 \mid a_2 \ox x_2 \ra  \coloneqq \la x_1 \mid \rho(a_1^* a_2) x_2 \ra_B,
\end{equation}
for all $a_1, a_2\in A$ and $x_1, x_2\in X_B$.
The \emph{KSGNS correspondence} $A \ox_\rho X$ is the completion of the quotient of $A \odot X$ by the ideal of tensors whose norm (coming from the semi-inner product) vanishes.
The product $A \ox_\rho X$ carries the obvious right $B$-module structure, and there is a left action by a $*$-homomorphism $\pi_\rho\colon A \to \End_B(A \ox_\rho X)$ satisfying
\begin{equation} 
\pi_\rho(a') (a \ox x) = a' a\ox x,
\end{equation}
for all $a', a\in A$ and $x\in X$,
which makes $(\pi_{\rho}, A\otimes_\rho X)$ a nondegenerate $A$--$B$-correspondence.
There is an adjointable map $V_\rho\colon X \to A \ox_\rho X$ satisfying  
\begin{equation} 
V_\rho(x) = \lim_i a_i \ox x, 
\end{equation}
where $x\in X$ and $(a_i)_i$ is an approximate unit for $A$.
Note that $V_\rho^*(a \ox x) = \rho(a) x$ for $a\in A$ and $x\in X$.
Then
\begin{equation} 
\rho(a) = V_\rho^* \pi_\rho(a) V_\rho, 
\end{equation}
for $a\in A$, and $\pi_\rho(A)V_\rho(X) \subseteq A \ox_\rho X$ is dense.

Moreover, if $(\phi,{}_A Y_B)$  is an $A$--$B$-correspondence and there exists an adjointable operator $W \colon X \to Y$ such that
$\rho(a) = W^* \phi(a) W$ for $a\in A$, and $\phi(A) W(X) \subseteq Y$ is dense, then 
\begin{equation} 
(\pi_{\rho},A \otimes_{\rho} X) \cong (\phi,Y )
\label{eq:KSGNS-uniqueness}
\end{equation}
as correspondences.
We refer to this property as the \emph{uniqueness} of the KSGNS construction.

Note that if $\rho$ is nondegenerate, then $V^*_\rho V_\rho(x)=\lim\rho(a_i)x=x$, so $V_\rho$ is an isometry. 
\end{construction}

\begin{example}\label{ex:KSGNS_cond_exp}
  An important case of the KSGNS construction is when $\rho \colon A \to B$ is a conditional expectation. 
  Define a degenerate $B$-valued inner product on $A$ by $\langle a_1 \mid a_2 \rangle_B = \rho(a_1^*a_2)$. 
  After quotienting by vectors of length zero and completing in the induced norm, we arrive at a $B$-module $L^2_B(A,\rho)$. 
  The identity map $\Id_A$ on $A$ induces an adjointable left action on $L^2_B(A,\rho)$ that we denote by $\Id_A^\rho$. 
  Then $(\pi_{\rho}, A \ox_{\rho} B) \cong (\Id_A^\rho, L^2_B(A,\rho))$.
\end{example}

Our aim is to construct a category similar to the Enchilada category of \cite{EKQR},
where instead of $*$-homomorphisms $\phi\colon A\to\End_B(X)$ we consider strict completely positive maps
$\rho\colon A \to\End_B(X)$.

\begin{defn}
  \label{def:pozzy-cozzy}
Let $A$ and $B$ be \mbox{$C^*$-algebras}.
A \emph{positive $A$--$B$-correspondence} is a pair $(\rho, {}_A X_B)$ where $X_B$ is a right $B$-module and $\rho \colon A \to \End_B(X_B)$ is a completely positive map.
The pair $(\rho, {}_A X_B)$ is \emph{strict} if $\rho$ is strict.

Two positive $A$--$B$-correspondences $(\rho, {}_A X_B)$ and $(\sigma, {}_A Y_B)$ are isomorphic 
if there is an adjointable unitary map $U\colon X\to Y$ such that $U\rho(a)U^*={\sigma}(a)$ for all $a\in A$. We note for later use that the coefficients are equal, not just isomorphic.
\end{defn}

Strictness of a positive correspondence ensures that the KSGNS construction yields a nondegenerate \mbox{$C^*$-correspondence} which is unique up to isomorphism~\mbox{\cite[Theorem 5.6]{Lance}}.

\begin{example}
Let $A$ be a \mbox{$C^*$-algebra.}  If $\phi\colon A \to \C$ is a state, then $(\phi, {}_{A} \C_{\C})$ is a strict positive correspondence.  
The KSGNS construction applied to this positive correspondence yields the GNS representation associated to $\phi$.  
\end{example}

\subsection{The Quesadilla semi-category}
\label{subsec:quesa}

The next lemma is key to defining a semi-category using strict completely positive maps. The reason we do not obtain a category is that the ``obvious'' identity elements only behave as identities from one side (see \cref{lem:all-fucked-up} below).

\begin{lemma}
  \label{lem:keyksgnslemma}
  Suppose that $(\rho, {}_A X_B)$ is a strict positive correspondence, and that $(\phi, {}_BY_C)$ is a \mbox{$C^*$-correspondence}.
  There is a strict completely positive map $\rho \ox \Id \colon A \to \End_C(X\ox_B Y)$ satisfying 
  \begin{equation}
    (\rho\ox \Id)(a)(x \ox y) = \rho(a)x \ox y,
    \label{eq:arr-id}
  \end{equation}
  for all $a\in A$, $x\in X$, and $y\in Y$, and an isomorphism of \mbox{$C^*$-correspondences},
  \begin{equation}
      (\pi_{\rho} \ox \Id_Y, (A \ox_\rho X) \ox_B Y )\cong (\pi_{\rho \ox \Id} , A \ox_{\rho \ox \Id} (X \ox_B Y)).
  \label{eq:iso-strict-mod}
  \end{equation}

\end{lemma}

\begin{proof}
Applying the KSGNS construction to $\rho$, we obtain the $A$--$B$-correspondence $A\otimes_\rho X$ with left action $\pi \coloneqq \pi_\rho\colon A \to \End_B(A\otimes_\rho X)$
and an adjointable operator $V \coloneqq V_\rho \colon X \to A\otimes_\rho X$ satisfying
$\rho(a) = V^*\pi(a)V$ and $\ol{\pi(A)V(X)} = A \ox_\rho X$, for all $a\in A$.
Since the $*$-homomorphism $\rho \ox \Id$ is the composition of $\rho$ with the nondegenerate $*$-homomorphism $\End_B(X) \to \End_C (X \ox_B Y)$ given by $T \mapsto T \ox \Id_{Y}$,
it follows that $\rho \ox \Id$ is a strict completely positive map.
So we also obtain the $A$--$C$-correspondence $(\pi_{\rho \ox \Id}, A \otimes_{\rho\otimes \Id}(X\otimes_B Y))$.

In order to prove the isomorphism~\labelcref{eq:iso-strict-mod}, we consider the adjointable operator 
$
V \ox \Id_Y \colon X \ox_B Y \to (A \ox_\rho X) \ox_B Y.
$
Note that $(V \ox \Id_Y)^* = V^* \ox \Id_Y$.
By the uniqueness of the KSGNS construction~\labelcref{eq:KSGNS-uniqueness}, it suffices to show that for all $a \in A$ we have
\begin{equation}
	\label{eq:ksgns_composition_sufficiency}
(\rho\ox\Id) (a) =(V^* \ox \Id_Y)\pi(a)(V \ox \Id_Y) \, \text{ and } \, \ol{\pi(A)(V \ox \Id_Y)(X \ox_B Y)} = (A \ox_\rho X) \ox_B Y.
\end{equation}

Fix an approximate unit $(a_i)_i$ for $A$.
Let $a\in A$, $x\in X$, and $y\in Y$. For the first equality of \labelcref{eq:ksgns_composition_sufficiency} we compute
\begin{align*}
  (V^* \ox \Id_Y)\pi(a)(V \ox \Id_Y) (x \ox y) & = \lim_i (V^* \ox \Id_Y) (aa_i \ox x) \ox y
                                                = V^*(a \ox x) \ox y \\                       
                                                &= \rho(a)x \ox y                               
                                                = (\rho \ox \Id) (x \ox y).
\end{align*}
For the second equality of \labelcref{eq:ksgns_composition_sufficiency} we compute
\[
  \pi(a) (V \ox \Id_Y) (x \ox y) = \lim_i (aa_i \ox x) \ox y =  (a \ox x) \ox y. 
\]
So $\pi(a) (V \ox \Id_Y) (X \ox_B Y)$ is dense in $(A \ox_{\rho} X) \ox_B Y$.
\end{proof}

If $A$ and $B$ are \mbox{$C^*$-algebras} and $(\rho, {}_A X_B)$ is a strict positive $A$--$B$-correpondence,
then we write $(\rho, {}_A X_B)\colon A \to B$ to indicate that we are thinking of it as a mapping, and below as a morphism, from $A$ to $B$.
\cref{lem:keyksgnslemma} allows us to define composition of strict positive correspondences.
Given two strict positive correspondences $(\rho, {}_A X_B) \colon A \to B$ and $(\sigma, {}_B Y_C) \colon B \to C$, we define 
\[
  X \ox_{{\sigma}} Y  \coloneqq X \ox_B (B \ox_{{\sigma}} Y)
\]
so that $(\rho \ox \Id_{B \ox_{{\sigma}} Y}, X \ox_{{\sigma}} Y) \colon A \to C$ is a strict positive correspondence. The KSGNS space $B \ox_{\sigma} Y$ is unique up to isomorphism since ${\sigma}$ is strict, and \labelcref{eq:arr-id} implies that
\[
\rho \ox \Id_{B \ox_{\sigma} Y}(a)(x \ox (b \ox y)) = \rho(a)x \ox (b \ox y)
\]
for all $a\in A$, $b\in B$, $x\in X$ and $y\in Y$.
To simplify notation we write
\[
(\rho, {}_A X_B) \ox (\sigma, {}_B Y_C) \coloneqq (\rho \ox \Id_{B \ox_{{\sigma}} Y}, X \ox_{{\sigma}} Y  ).
\]

The following result gives us associativity of composition of isomorphism classes of positive correspondences. 
\begin{lemma}
Let $(\rho, {}_A X_B) \colon A \to B$, $(\sigma, {}_B Y_C) \colon B \to C$, and $(\tau,{}_C Z_D) \colon C \to D$ be strict positive correspondences. 
Then
\[
\Big( (\rho, {}_A X_B) \ox (\sigma, {}_B Y_C) \Big) \ox (\tau,{}_C Z_D) \cong (\rho, {}_A X_B) \ox \Big(	(\sigma, {}_B Y_C) \ox (\tau,{}_C Z_D)\Big).
\]
\end{lemma}

\begin{proof}
 The compositions
  \[
      \Big( (\rho, {}_A X_B) \ox (\sigma, {}_B Y_C) \Big) \ox (\tau,{}_C Z_D) 
      = \Big((\rho \ox \Id_{B\ox_{\sigma}Y} )\ox \Id_{C \ox_T Z}, {}_{A} (X \ox_B (B \ox_{\sigma} Y)) \ox_C (C  \ox_T Z)\Big)
  \]
  and
  \[
      (\rho, {}_A X_B) \ox \Big(	(\sigma, {}_B Y_C) \ox (\tau,{}_C Z_D)\Big)  
      = \Big(\rho\ox \Id_{(B \ox_{\sigma} Y) \ox_C ( C \ox_T Z)} ,{}_{A} X \ox_B ( (B \ox_{\sigma} Y) \ox_C ( C \ox_T Z) )\Big)
  \]
  are isomorphic by two applications of \cref{lem:keyksgnslemma}. 
\end{proof}

With associative composition of morphisms, it seems obvious that the next step is to define a positive correspondence category. Sadly, not quite. 

\begin{lemma}
\label{lem:all-fucked-up}
Let $(\rho,{}_AX_B)$ be a strict positive correspondence. Then
\[
(\rho,{}_AX_B)\ox({\rm Id},{}_BB_B)\cong(\rho,{}_AX_B)
\quad
\text{but}
\quad
({\rm Id},{}_AA_A)\ox(\rho,{}_AX_B)\cong (\pi_\rho,A\ox_\rho X).
\]
\end{lemma}
\cref{lem:all-fucked-up} says that typically we do not have identity morphisms, though we do always have right identities. 
The usual names for a `category' without identities is a \emph{semi-category} or a \emph{semi-groupoid} \cite[Appendix B]{Tilson}. 
Though, having right identities gives us more structure than a general semi-category.
For instance, we can still work just with morphisms by identifying objects with their right identities, though we will not explore the ramifications of the semi-category structure here.

Following the Tex-Mex precedent of \cite{EKQR} we introduce the following.
\begin{defn}
  \label{def:semi-category}
The \emph{positive correspondence semi-category} (or the \emph{Quesadilla semi-category}) $\Quesa$ is defined as follows:
  \begin{enumerate}
    \item Objects of $\Quesa$ are \mbox{$C^*$-algebras}.
      \item Morphisms from $A$ to $B$ are isomorphism classes of strict positive $A$--$B$-correspondences.
      \item Given morphisms $(\rho, {}_A X_B) \colon A \to B$ and $(\sigma, {}_B Y_C) \colon B \to C$, composition is defined by
            \[
              (\rho, {}_A X_B) \ox (\sigma, {}_B Y_C) \coloneqq (\rho \ox \Id_{B \ox_{{\sigma}} Y},{}_{A} X \ox_{{\sigma}} Y_C  )\colon A \to C.
            \]

  \end{enumerate}
\end{defn}
\begin{rmk}
  Using positive correspondences as morphisms instead of their isomorphism classes, we could perhaps define 2-morphisms to obtain a bi-semi-category analogous to the $C^*$-correspondence bi-category of \cite{Meyer-Sehnem}. 
\end{rmk}


\begin{example}
  If $\rho \colon A \to B$ is a strict completely positive map, then the isomorphism class of $(\rho,{}_A B_B)$ is a morphism in $\Quesa$ from $A$ to $B$.
  If $\rho$ is a $*$-homomorphism then $(\rho,{}_A B_B)$ is the isomorphism class of the \mbox{$C^*$-correspondence} associated with $\rho$.
\end{example}

\begin{example}
	Let $\rho \colon A \to B$ be a conditional expectation with nondegenerate inclusion $\iota \colon B \to A$.
    We consider the positive correspondence $(\rho,{}_AB_B)$ and the \mbox{$C^*$-correspondence} $(\iota,{}_BA_A)$. 
	There are two possible compositions,
	\[
	(\rho\ox{\rm Id},{}_AB \ox_{\iota } A_A) 
	\cong (\iota\circ\rho,{}_AA_A)
	\quad\mbox{and}\quad (\iota\ox{\rm Id}_B,{}_BA\ox_{\rho}B_B) \cong (\iota,{}_BL^2_B(A,\rho)).
	\]
	The latter correspondence is that of \cref{ex:KSGNS_cond_exp}, with the left action restricted to $B$. 
	As $\iota\circ \rho$ is always an adjointable projection on $L^2_B(A,\rho)$, 
	there is a copy of the identity correspondence $({\rm Id}_B,{}_BB_B)$ contained as a complemented sub-correspondence in $(\iota,{}_BL^2_B(A,\rho))$.
\end{example}

The positive correspondence semi-category seems to be the right domain of the KSGNS construction. Since every \mbox{$C^*$-correspondence} is in particular a strict positive correspondence, $\Ench$ is a wide subcategory of $\Quesa$. 
Further, we may interpret the proposition below as saying that $\Ench$ is a retract of $\Quesa$.

\begin{prop}
  \label{prop:ksgnsfunctor}
  The $\KSGNS$ construction provides a semi-functor\footnote{i.e. preserves composition.}
  \[
      \KSGNS \colon \Quesa \to \Ench 
  \]
  such that:
  \begin{enumerate}
    \item $\KSGNS(A) = A$ for every \mbox{$C^*$-algebra} $A$; and
    \item for a positive correspondence $(\rho, {}_A X_B)$, $\KSGNS(\rho, {}_A X_B)$ is the class of the correspondence $(\pi_\rho,{}_A(A \ox_\rho X)_B)$. In particular, if $\rho$ is a nondegenerate $*$-homomorphism, then $\KSGNS(\rho, {}_A X_B)=(\rho, {}_A X_B)$.
  \end{enumerate}
 If $U \colon \Ench \to \Quesa$ is the forgetful semi-functor given by considering nondegenerate $*$-homomorphims as strict completely positive maps, then
  \[
  \KSGNS{} \circ U = \Id_\Ench 
  \quad \text{and} \quad
  U \circ \KSGNS{}~\text{is idempotent}.
  \]
\end{prop}

\begin{proof}
It is clear from our construction that morphisms in $\Quesa$ are sent to morphisms in $\Ench$,
and the fact that the $\KSGNS$ respects composition is the content of~\cref{lem:keyksgnslemma}.
If $(\rho, {}_A X_B)$ is a strict positive correspondence and $\rho$ is a nondegenerate $*$-homomorphism, then there is an isomorphism $(\pi_\rho,{}_A(A \ox_\rho X)_B) \cong (\rho, {}_A X_B)$ of \mbox{$C^*$-correspondences},
so $\KSGNS$ acts as the identity. Since morphisms in $\Ench$ and $\Quesa$ are isomorphism classes, the final assertions are now clear.
\end{proof}

In a general semi-category we cannot talk about invertible morphisms, as we lack identities. 
In $\Quesa$, we may talk about invertible morphisms since we have right identities.

\begin{defn}
A morphism $(\rho,{}_AX_B)$ in $\Quesa$ is \emph{invertible} if there exists a morphism $(\sigma,{}_BY_A)$ such that
\[
(\rho,{}_AX_B)\ox(\sigma,{}_BY_A)=({\rm Id},{}_AA_A)  \quad\mbox{and}\quad (\sigma,{}_BY_A)\ox(\rho,{}_AX_B)=({\rm Id},{}_BB_B).
\]
Here, $({\rm Id},{}_AA_A)$ and $({\rm Id},{}_BB_B)$ are the right identity morphisms for $A$ and $B$.
\end{defn}

The lemma below should be compared to \cite[Lemma 2.4]{EKQR} which says that the invertible morphisms in $\Ench$ are the Morita equivalence bimodules.
\begin{lemma}
The invertible morphisms in $\Quesa$ are exactly the Morita equivalence bimodules. 
\end{lemma}

\begin{proof}
Morita equivalence bimodules are invertible with inverse given by the conjugate module.
Conversely, suppose that $(\rho, {}_A X_B)\colon A \to B$ and $(\sigma, {}_BY_A) \colon B \to A$ represent mutually inverse morphisms in $\Quesa$. 
Then
\begin{align*}
  (\rho, {}_A X_B) 
  &\cong (\rho, {}_A X_B) \ox (\Id_B,{}_B B_B)
  \cong 
  (\rho, {}_A X_B) \ox (\sigma, {}_B Y_A) \ox (\rho, {}_A X_B)\\
  &\cong (\Id_A, {}_A A_A) \ox  (\rho, {}_A X_B)
  \cong(\pi_\rho,A \ox_\rho X).
\end{align*}
Uniqueness of the KSGNS construction implies that $\rho$ must be a $*$-homomorphism so $(\rho, {}_A X_B)$ is a \mbox{$C^*$-correspondence}.
A similar argument shows that $({\sigma},{}_BY_A)$ is also a $C^*$-correspondence.
By~\cite[Lemma 2.4]{EKQR}, the invertible morphisms in $\Ench$ are imprimitivity bimodules, so $(\rho, {}_A X_B)$ is an imprimitivity bimodule,
and the result follows.
\end{proof}

Under certain conditions, the composition in $\Quesa$ simplifies significantly. 

\begin{lemma}\label{lem:compose-expectation}
Let $\rho \colon A \to B$ and ${\sigma} \colon B \to C$ be strict completely positive maps between \mbox{$C^*$-algebras} and assume that $\rho$ is a conditional expectation.
There is an inclusion of $C$-modules 
\[
  \psi \colon (A \ox_\rho B) \ox_B (B \ox_{\sigma} C) \to A \ox_{{\sigma}\circ \rho} C.
\]
satisfying $\psi\big((a \ox b_1) \ox (b_2 \ox c)\big) = ab_1 b_2 \ox c$, for all $a\in A$, $b_1, b_2\in B$, and $c\in C$.
If $B$ contains an approximate unit for $A$, then $\psi$ is an isomorphism. 
\end{lemma}

\begin{proof}
Fix $a \in A$, $b_1,b_2 \in B$, and $c \in C$, and observe that
\begin{align*}
  \la{(a \ox b_1) \ox (b_2 \ox c) \mid (a \ox b_1) \ox (b_2 \ox c)}\ra_C   
  & = c^*{\sigma}(b_2^*b_1^*\rho(a^*a)b_1b_2)c                        
   =  c^*{\sigma}\circ \rho((a b_1 b_2)^* a b_1 b_2)c      \\
  & = \la ab_1b_2 \ox c \mid ab_1b_2 \ox c\ra_C,
\end{align*}
where the first inner product is on $(A \ox_\rho B) \ox_B (B \ox_{\sigma} C)$ and the second is on $A \ox_{{\sigma}\circ \rho} C$.
It follows that there is a well-defined isometric linear map 
$\psi \colon (A \ox_\rho B) \ox_B (B \ox_{\sigma} C)\to A \ox_{{\sigma}\circ \rho} C$ satisfying
$
\psi((a \ox b_1) \ox (b_2 \ox c)) = ab_1b_2 \ox c.$
It is straightforward to see that $\psi$ respects the left action of $A$ and right action of $C$.

For the second statement, choose an approximate unit $(b_i)_i$ for $A$ contained in $B$.
Then for fixed $a\in A$ and $c\in C$, 
\[
  \lim_i \psi\big((a \ox b_i^{1/2}) \ox (b_i^{1/2} \ox c)\big) = \lim_i a b_i \ox c = a \ox c,
\]
and it follows that $\psi$ is surjective.
\end{proof}

\begin{example}
We give two examples of expectations arising from the Cuntz--Pimsner algebra of a \mbox{$C^*$-correspondence} $(\phi_X,{}_AX_A)$. 
Let $\gamma \colon \bT \to \Aut(\cO_X)$ denote the gauge action and let $\cO_X^{\gamma}$ denote the associated fixed-point algebra. 
There is a conditional expectation $\Phi \colon \cO_X\to\cO_X^\gamma$ given by averaging over $\gamma$. 
Applying the $\KSGNS$ construction yields the correspondence $(\pi_\Phi,{}_{\cO_X}L^2_{\cO_X^\gamma}(\cO_X,\Phi))$ underlying the 
Kasparov module encoding the gauge action, \cite{PR,CNNR}. 

If $X$ is a bi-Hilbertian bimodule of finite right Watatani index (see \cref{sec:bi-curious}), 
there is another expectation
$\Phi_\infty:\cO_X\to A$. Applying the $\KSGNS$ functor yields the correspondence $(\pi_{\Phi_\infty}, {}_{\cO_X}\Xi_A)$ underlying the Kasparov module representing the class of the defining extension	\labelcref{eq:def-ext}, \cite{RRSext,GMR}.  We note that both $A$ and $\cO_X^\gamma$ contain an approximate unit for $\cO_X$.

Given a densely-defined norm lower semi-continuous trace $\tau$ on $\cO_X^\gamma$ or on $A$, we can then compose with the GNS representations $(\pi_\tau,{}_{\cO_X^\gamma}L^2(\cO_X^\gamma,\tau))$ or $(\pi_\tau,{}_AL^2(A,\tau))$ using \cref{lem:compose-expectation}, and apply the KSGNS construction to obtain 
\[
(\pi_{\tau\circ\Phi},{}_{\cO_X}L^2(\cO_X,\tau\circ\Phi))\quad\mbox{or}\quad
(\pi_{\tau\circ\Phi_\infty},{}_AL^2(A,\tau\circ\Phi_\infty)).
\]	
The process of identifying KMS states of quasi-free real actions on $\cO_X$ outlined by Laca and Neshveyev is compatible with these decompositions, \cite{LN04,GRU,RRSext}.
\end{example}

To understand the composition of morphisms in $\Quesa$ to and from Cuntz--Pimsner algebras, we need to understand how composition in $\Quesa$ interacts with quotient maps. 
Recall that for a quotient map $q \colon A \to A/I$ with linear splitting $s \colon A/I \to A$ we have $s(ab) - s(a)s(b) \in I$ and $s(a^*) - s(a)^* \in I$,
for all $a,b\in A/I$.

\begin{lemma}
\label{lem:quotient-postitive-map}
Suppose that we have a commuting diagram of \mbox{$C^*$-algebras}
\[
    \begin{tikzcd}
        0 \arrow[r] & I \arrow[r] \arrow[d, "\rho|_I"]  & A \arrow[d, "\rho"] \arrow[r,"q_A"] & A/I \arrow[r]  & 0 \\
        0 \arrow[r] & J \arrow[r]           & B \arrow[r, "q_B"]            & B/J \arrow[r]                   & 0
    \end{tikzcd}
\]
 with exact rows of $*$-homomorphisms. Suppose that the top row has a completely positive splitting, and that $\rho$ is a strict completely positive  map satisfying $\rho(I)\subseteq J$. Then the map $\wt{\rho} \colon A/I \to B/J$ defined by $\wt{\rho}(a+I)=\rho(a)+J$ is completely positive and strict. 
\end{lemma}
\begin{proof}
	Fix a splitting $s \colon A/I \to A$ of $q_A$. Then $\tilde{\rho} = q_B \circ \rho \circ s$. Complete positivity of $\tilde{\rho}$ is clear, so we show strictness.
Since $\rho$ is strict we can fix an approximate unit $(e_i)$ of $A$ such that $\rho(e_i)$ is strictly Cauchy in $\Mult(B)$. 
Then $(q_A(e_i))$ is an approximate unit for $A/I$ and $s \circ q_A(e_i) = e_i + k_i$, for some $k_i \in I$.  
As $\rho(I) \subseteq J$, we have $\wt{\rho}(q_A(e_i)) = q_B \circ \rho(e_i)$. 
Since $q_B$ is a surjective $*$-homomorphism it follows that $\wt{\rho}(q_A(e_i))$ is strictly Cauchy, so $\wt{\rho}$ is strict.
\end{proof}

In the case where $\rho \colon A \to B$ is a conditional expectation, the KSGNS space of the induced completely positive map $\wt{\rho} \colon A/I \to B/J$ enjoys a construction similar to that of \cref{ex:KSGNS_cond_exp}.
\begin{lemma}
\label{lem:quotient-expectation}
Suppose that we have the commuting diagram 
\[
    \begin{tikzcd}
        0 \arrow[r] & I \arrow[r] \arrow[d, "\rho|_I"]  & A \arrow[d, "\rho", shift left] \arrow[r,"q_A"] & A/I \arrow[r] \arrow[d,"\wt{\rho}"]  & 0 \\
        0 \arrow[r] & J \arrow[r]           & B \arrow[r, "q_B"]   \arrow[u, "\alpha", shift left]          & B/J \arrow[r]                   & 0
    \end{tikzcd}
\]
as in \cref{lem:quotient-postitive-map}. Suppose 
further that $\rho \colon A \to B$ is a conditional expectation with nondegenerate inclusion $\alpha \colon B \to A$. Let $\big({\rm Id}_{A/I},L^2_{B/J}(A/I,\wt{\rho})\big)$ be the completion (after quotienting by zero vectors) of $A/I$ in the norm coming from the semi-inner product
\begin{equation}
	\langle a_1+I\mid a_2+I\rangle_{B/J}=\rho(a_1^*a_2)+J,\qquad a_1,a_2\in A.
	\label{eq:inner-product}
\end{equation}
Then 
$
\KSGNS(\wt{\rho},{}_{A/I}(B/J)_{B/J}) \cong ({\rm Id}_{A/I},L^2_{B/J}(A/I,\wt{\rho})).
$
\end{lemma}
\begin{proof}
Apart from linearity, it is easy to check that \labelcref{eq:inner-product} defines an inner product. Fix a completely positive splitting $s_A$ of $q_A$ so that $\wt{\rho} = q_B \circ \rho \circ s_A$. For the $B/J$-linearity,
let $a+I\in A/I$ and $b+J\in B/J$, and compute
\begin{align*}
\wt{\rho}(a+I)(b+J)&=\rho\circ s_A(a+I)b+J
=\rho(a+i)b+J\quad \mbox{for some }i\in I\\
&=\rho(a\alpha(b)+i\alpha(b))+J
=\rho(a\alpha(b))+J.
\end{align*}
A straightforward computation shows that
\[
\langle a_1+I\mid a_2+I\rangle_{B/J}(b+J)=\langle a_1+I\mid a_2\alpha(b)+I\rangle_{B/J}.
\]
For $a\in A$ and $b\in B$ we have the equality
\begin{align*}
\langle [a\alpha(b)+I]\mid[a\alpha(b)+I]\rangle_{B/J}
&=\rho(\alpha(b^*)a^*a\alpha(b))+J
=(b^*+J)(\rho(a^*a)+J)(b+J)\\
&=\langle b+J\mid \wt{\rho}(a^*a+I)(b+J)\rangle_{B/J}\\
&=\langle b+J\mid \wt{\rho}(\langle a+I\mid a+I\rangle_{A/I})(b+J)\rangle_{B/J}\\
&=\langle (a+I)\ox(b+J)\mid (a+I)\ox(b+J)\rangle_{B/J},
\end{align*}
so there is a well-defined isometric linear map
$
\Lambda\colon A/I\ox_{\wt{\Psi}}B/J\to L^2_{B/J}(A/I,\wt{\rho})$ satisfying $\Lambda((a+I)\ox(b+J))\coloneqq[a\alpha(b)+I]$. The right $B/J$-linearity of $\Lambda$ is routine and nondegeneracy of $\alpha$ gives the surjectivity of $\Lambda$.
\end{proof}

The following result is analogous to \cref{lem:compose-expectation} once we pass to quotients. The result is needed in \cref{sub:cozzy-proj}.

\begin{lemma} 
\label{lem:compose-expectation-quot}
Suppose that 
\[
  \begin{tikzcd}
      0 \arrow[r] & I \arrow[r] \arrow[d, "\rho|_I"]  & A \arrow[d, "\rho",shift left] \arrow[r,"q_A"] & A/I \arrow[r] \arrow[d , "\wt{\rho}"]   & 0 \\
      0 \arrow[r] & J \arrow[r]  \arrow[d, "\sigma|_J"]  & B \arrow[r, "q_B"] \arrow[d, "\sigma"]   \arrow[u, "\alpha", shift left]       & B/J \arrow[r] \arrow[d , "\wt{\sigma}"]                  & 0\\
      0 \arrow[r] &K \arrow[r] & C \arrow[r, "q_C"] & C/K \arrow[r] & 0
  \end{tikzcd},
\]
is a commuting diagram of $C^*$-algebras
with exact rows of $*$-homomorphisms that have completely positive splittings $s_A \colon A/I \to A$ and $s_B \colon B/J \to B$. Let $\rho \colon A \to B$ be a conditional expectation with corresponding inclusion $\alpha \colon B \to A$,
and let $\sigma \colon B \to C$ be a strict completely positive map.
Further suppose that $\rho(I) \subseteq J$ and $\sigma(J) \subseteq K$, and that the maps $\wt{\rho}$ and $\wt{\sigma}$ are the strict completely positive maps induced by \cref{lem:quotient-postitive-map}.
Then there is an isometric inclusion of correspondences
\[
\psi \colon (A/I \ox_{\wt{\rho}} B/J) \ox_{B/J} (B/J \ox_{\wt{\sigma}} C/K) \to A/I \ox_{\wt{\sigma}\circ\wt{\rho}} C/K
\]
satisfying 
\begin{equation}\label{eq:composition_quotient_inclusion}
	\psi\big((a \ox b_1) \ox (b_2 \ox c)\big) =  a q_A( \alpha (s_B(b_1b_2))) \ox c,
\end{equation}
for $a\in A/I$, $b_1,b_2\in B/J$, and $c\in C/K$. 
If $q_A \circ \alpha \circ s_B(B/J)$ contains an approximate unit for $A/I$, then $\psi$ is also surjective.
\end{lemma}

\begin{proof}
We first show that \labelcref{eq:composition_quotient_inclusion} yields a well-defined map.
Let $s_A \colon A/I \to A$ and $s_B \colon B/J \to J$ be the completely positive splittings of $q_A$ and $q_B$, respectively.
For $a \in A/I$, $b_1,b_2 \in B/J$, and $c \in C/K$, a computation similar to that of~\cref{lem:compose-expectation} shows that
\begin{align*}
    \big\langle(a \ox b_1) \ox (b_2 \ox c) \mid (a \ox b_1) \ox (b_2 \ox c)\big\rangle_{C/K} & = c^* \wt{\sigma} \big(b_2^*b_1^* \wt{\rho}(a^*a) b_1b_2\big)c.
\end{align*}
Recall from \cref{lem:quotient-postitive-map} that $\wt{\rho}  = q_B \circ \rho \circ s_A$ and $\wt{\sigma} = q_C \circ \sigma \circ s_B$. 
Since $s_A(a^*)s_A(a) - s_A(a^*a) \in I$ and $\rho(I) \subseteq J$ it follows that
  \begin{align*}
      b_2^*b_1^* \wt{\rho}(a^*a) b_1b_2 & = q_B \big( s_B(b_2^*b_1^*) \rho(s_A(a^*)s_A(a))  s_B(b_1b_2) \big)                     \\
                                     & = q_B \circ \rho \big( \alpha( s_B(b_1b_2))^*s_A(a)^*s_A(a)  \alpha(s_B(b_1b_2)) \big).
  \end{align*}
  Define $d \coloneqq a q_A(\alpha(s_B(b_1b_2))) \in A/I$ and note that 
  \begin{equation*}
    s_A(d^*d) -   \alpha(s_B(b_1b_2))^*s_A(a)^*s_A(a) \alpha(s_B(b_1b_2)) \in I. 
  \end{equation*}
  Consequently,
  \begin{align*}
      \la(a \ox b_1) \ox (b_2 \ox c) \mid (a \ox b_1) \ox (b_2 \ox c)\ra_C & = c^* \wt{\sigma} ( q_B \circ \rho \circ s_A(d^*d)) c 
      = c^* \wt{\sigma} \wt{\rho} (d^*d) c                  \\
   & = \la d \ox c \mid d \ox c \ra_{C/K},
  \end{align*}
  where the latter inner product is on $A/I \ox_{\wt{\sigma}\circ\wt{\rho}} C/K$.
  This shows that $\psi$ is a well-defined isometric linear map.
  It is straightforward to see that $\psi$ respects the left action of $A/I$ and the right action of $C/K$.

  For the final statement, suppose $(b_i)_i$ is a net in $B/J$ such that $a_i \coloneqq q_A \circ \alpha \circ s_B(b_i)$ defines an approximate unit for $A/I$.
  Then
  \[
    \lim_i \psi( (a \ox b_{i}^{1/2}) \ox (b_i^{1/2}  \ox c)) = \lim_i a a_i \ox c = a \ox c,
  \]
  and it follows that $\psi$ is surjective.
\end{proof}

\section{Cuntz--Pimsner morphisms from projections}
\label{sec:projections}

Armed with the theoretical framework of positive correspondences, we now look at applications to Cuntz--Pimsner algebras. 
We begin by outlining the difficulty with defining correspondence morphisms from complemented sub-correspondences. 
We then turn to more general projections on Fock modules and the compressions of product systems (over $\N$) which give rise to subproduct systems (over $\N$). 

\subsection{Morphisms from correspondence projections}
\label{sub:cozzy-proj}

A \emph{complemented} sub-correspondence is a \mbox{$C^*$-correspondence} that appears as a direct summand of a $C^*$-correspondence.
This is stronger than being complemented as a right $C^*$-module since the left action must also be respected.
Not every sub-correspondence is complemented. A sub-correspondence is complemented precisely if it is the image of a correspondence projection. 
\begin{defn}
  A \emph{correspondence projection} on a correspondence $(\phi,{}_A {X}_B)$ is a projection $P \in \End_B(X)$ such that $\phi(a)P = P\phi(a)$ for all $a \in A$.
\end{defn}

If $P$ is a correspondence projection on $(\phi,{}_AX_B)$, then $1-P$ is a correspondence projection,
and $(P\phi, {}_A PX_B)$ is a complemented sub-$A$--$B$-correspondence determining the direct sum decomposition $X_B \cong PX_B \oplus (1-P)X_B$.
Conversely, every complemented sub-$A$--$B$-correspondence is the image of a correspondence projection. 
If $\cor{\phi_Y}{A}{Y}{A}$ is the image of a correspondence projection $P$ on $\cor{\phi_X}{A}{X}{A}$ we will write $P \colon \cor{\phi_X}{A}{X}{A} \to \cor{\phi_Y}{A}{Y}{A}$
where $Y_A = PX_A$. 
We will also write $Y_A^\perp \coloneqq (1-P)X_A$.

If $Y = PX$, then the left action $\phi_X$ acts diagonally on $X$ with respect to the direct sum decomposition.
This follows from the computation 
\begin{align*}
  \phi_X(a) x & = P\phi_X(a) P x + (1-P) \phi_X(a) Px + P \phi_X(a) (1-P)x + (1-P) \phi_X(a) (1-P)x \\
  & =\phi_Y(a)Px + \phi_{Y^{\perp}}(a) (1-P)x,
\end{align*}
which holds for all $a\in A$ and $x\in X$.
In particular, $\phi_X(a) \in \End^0_{A}(X)$ if and only if $\phi_Y(a) \in \End^0_{A}(Y) $ and $\phi_{Y^{\perp}}(a) \in \End^0_{A}(Y^{\perp})$.

Also observe that $\ker(\phi_X) = \ker(\phi_Y) \cap \ker(\phi_{Y^{\perp}})$. 
Since $\ker(\phi_Y)^{\perp} \cap \ker(\phi_{Y^\perp})^{\perp}$ is a subset of $(\ker(\phi_Y) \cap \ker(\phi_{Y^{\perp}}))^{\perp}$  the covariance ideals satisfy $J_Y \cap J_{Y^{\perp}} \subseteq J_X$. 
In general, it is not true that $J_Y \subseteq J_X$ nor that $J_X \subseteq J_Y$.
However, in the case that $(\phi_{Y^{\perp}}, Y_A^{\perp})$ is regular, we have the following immediate result.

\begin{lemma}\label{lem:nice_complement_good_ideals}
  If $\cor{\phi_Y}{A}{Y}{A}$ is a complemented sub-correspondence of $\cor{\phi_X}{A}{X}{A}$, and $(\phi_{Y^{\perp}}, Y_A^{\perp})$ is regular, then $J_X = J_Y$.
\end{lemma}

Let $\cor{\phi_Y}{A}{Y}{A}$ is a complemented sub-correspondence of $\cor{\phi_X}{A}{X}{A}$, and let $\iota \colon Y \to X$ denote the inclusion. 
Then $(\Id_A,\iota)$ is a correspondence morphism,
so the universal property of $\cT_Y$ yields an injective $*$-homomorphism 
\begin{equation}
\alpha \coloneqq  \Id_A \times \iota \colon \cT_Y \to \cT_X \subseteq \End_A(\Fock_X)
\label{eq:alpha}
\end{equation}
satisfying $\alpha(T_\xi T_\eta^*) = T_{\iota(\xi)}T_{\iota(\eta)}^*$ for all $\xi, \eta \in Y$. 
The correspondence morphism $(\Id_A,\iota)$ is typically not covariant as the following example shows,
so $\alpha \colon \cT_Y \to \cT_X$ does not typically descend to a $*$-homomorphism between Cuntz--Pimsner algebras.

\begin{example}
	\label{ex:covariance_fails}
Let ${}_A X_A = {}_{\C} \C^2_\C$ with the left action given by scalar multiplication. Let ${}_A Y_A = {}_\C \C_\C$ be the complemented sub-$\C$--$\C$-correspondence spanned by the basis vector $e_1$ of $\C^2$.
Let $T$ denote the generating isometry $j_Y(e_1)$ of $\cT_{Y} \cong \cT$ and let $T_1 = j_X(e_1)$ and $T_2 = j_X(e_2)$ be the generating isometries of $\cT_X \cong \cT_2$. Then $\alpha(T) = T_1$.
  
In the quotient Cuntz--Pimsner algebra, $T$ descends to the unitary $U \colon z \mapsto z$ that generates $\cO_Y \cong C(S^1)$, 
while $T_1$ and $T_2$ descend to isometries $S_1$ and $S_2$ generating $\cO_X \cong \cO_2$. 
Consequently, $\alpha$ cannot descend to a $*$-homomorphism between Cuntz--Pimsner algebras as such a map would take the unitary $U$ to the non-unitary isometry $S_1$.  
\end{example}

Aside from $\alpha$, there is a second representation of $\cT_Y$ in  $\End_A(\Fock_X)$ arising from the correspondence projection $P$ that we will take a moment to explain. 
First, observe that if $\cor{\phi_Y}{A}{Y}{A}$ is a sub-$A$--$A$-correspondence of $\cor{\phi_X}{A}{X}{A}$, then the Fock module $\Fock_Y$ may be identified as a sub-$A$--$A$-correspondence of $\Fock_X$. The following lemma shows that complementability is also preserved.

\begin{lemma}
  \label{lem:fock-projection}
Suppose that $P_0 \colon \cor{\phi_X}{A}{X}{A} \to \cor{\phi_Y}{A}{Y}{A}$ is a correspondence projection. 
Then $P_0$ extends to a correspondence projection $P_n \colon \cor{\phi_X}{A}{X^{\ox n}}{A} \to \cor{\phi_Y}{A}{Y^{\ox n}}{A}$ by
\begin{equation}
    \label{eq:correspondence_projection_n}
    P_n \coloneqq P_0 \ox \cdots \ox P_0.
\end{equation}
The sum  $\sum_{n \ge 0} P_n$ converges strictly to a projection $P \colon (\phi_X,\Fock_X) \to (\phi_Y,\Fock_Y)$, so 
 $\Fock_Y$ is complemented in $\Fock_X$.
\end{lemma}

\begin{proof}
The $A$-bilinearity of $P_0$ ensures that \labelcref{eq:correspondence_projection_n} is a well-defined correspondence projection. 
Since $\Fock_X$ is a direct sum, the second statement follows from the first.
\end{proof}

If $\Fock_Y$ is complemented in $\Fock_X$, then there is an inclusion $\overline{\iota} \colon \End_A(\Fock_Y) \to \End_A(\Fock_X)$ 
defined by $\ol{\iota}(T)(x \oplus x^{\perp}) = Tx$ for $x\in \Fock_Y$ and $x^{\perp} \in \Fock_Y^{\perp}$. Since $\cT_Y$ is a subalgebra of $ \End_A(\Fock_Y)$, the map $\ol{\iota}$ restricts to an injective $*$-homomorphism 
\begin{equation}
\beta \colon \cT_Y \to \End_A(\Fock_X).
\label{eq:beta}
\end{equation}

The representations $\alpha$ and $\beta$ of $\cT_Y$ in $\End_A(\Fock_X)$ do not usually coincide.
Indeed, if $A$ were unital then $\alpha(1_A) = \Id_{\Fock_X}$ while $\beta(1_A) = \ol{\iota}(\Id_{\Fock_Y})$ is typically a proper subprojection of $\Id_{\Fock_X}$. 
Moreover, if $T_y\in\cT_Y$ is a creation operator on $\Fock_Y$ for some $y \in Y_A$, then $\beta(T_y)$ is typically not a creation operator on $\Fock_X$ since it is zero on $\Fock_Y^{\perp}$.

Despite neither $\alpha$ nor $\beta$ inducing a $*$-homomorphism from $\cO_Y$ to $\cO_X$, we can use the projection $P$ of \cref{lem:fock-projection} to build a conditional expectation from $\cT_X$ to $\cT_Y$ which respects both inclusions $\alpha$ and $\beta$. By then applying \cref{lem:quotient-expectation} we will construct a positive map from $\cO_X$ to $\cO_Y$ which gives morphisms in both $\Quesa$ and $\Ench$.

\begin{lemma}
\label{lem:fock-expectation}
Let $P \colon (\phi_X,\Fock_X) \to (\phi_Y,\Fock_Y)$ be the correspondence projection of \cref{lem:fock-projection}. The map $\Psi_P \colon \End_A(\Fock_X) \to \End_A(\Fock_Y)$ defined by 
\begin{equation}\label{eq:psi_P}
\Psi_P(T) = PTP
\end{equation}
is a conditional expectation. It restricts to a conditional expectation $\Psi_P \colon \cT_X \to \cT_Y$ such that
\begin{equation}\label{eq:toeplitz-expectation}
  \Psi_P (\alpha(b_1)a \alpha(b_2)) = b_1\Psi_P( a ) b_2 = \Psi_P (\beta(b_1)a \beta(b_2))
\end{equation}
for all $a \in \cT_X$ and $b_1,\,b_2 \in \cT_Y$.
\end{lemma}

\begin{proof}
There are multiple Fock spaces and Toeplitz algebras around, so if $T \in \cT_X$, then we write $T^X$ for the corresponding operator on $\Fock_X$, and if $T \in \cT_Y$ we write $T^Y$ for the corresponding operator on $\Fock_Y$. We also use $\iota \colon \Fock_Y \to \Fock_X$ to identify $\Fock_Y$ as a submodule of $\Fock_X$. In particular, if $\xi \in Y$, then $\alpha(T_\xi^Y) = T^X_{\xi}$.

That \labelcref{eq:psi_P} defines an expectation for the inclusion $\ol{\iota}  \colon \End_A(\Fock_Y) \to \End_A(\Fock_X)$ is clear. It remains to show that $\Psi_P$ restricts to Toeplitz algebras. 
Let $y = y_1 \ox \cdots \ox y_k \in \Fock_Y \subseteq \Fock_X$ and $\xi , \eta \in \Fock_X$ be elementary tensors. 
It follows from the $A$-bilinearity of $P$ that for $k \ge 1$
\begin{align*}
  \Psi_P(T^X_\xi T^{X*}_\eta) y
  &= PT^X_\xi T^{X*}_\eta P y\\
  & = \begin{cases}
      P(\xi \cdot\langle \eta \mid Py_1 \ox \cdots \ox y_{|\eta|} \rangle_A  \ox y_{|\eta|+1} \ox \cdots \ox y_{k}) & \text{if } k \geq |\eta| \\
      0                                                                                            & \text{if } k <  |\eta|
    \end{cases}   \\
                             & = \begin{cases}
        (P\xi )\cdot \langle P \eta \mid y_1 \ox \cdots \ox y_{|\eta|} \rangle_A  \ox y_{|\eta|+1} \ox \cdots \ox y_{k} & \text{if } k \geq |\eta| \\
        0                                                                                             & \text{if } k < |\eta|
    \end{cases}   
                             = T_{P\xi}^Y T_{P\eta}^{Y*} y.
\end{align*}
Similarly, if $y \in A$, then $\Psi_P(T^X_\xi T^{X*}_\eta)y = T_{P\xi}^Y T_{P\eta}^{Y*} y$, so $\Psi_P (\cT_X) = \cT_Y$. 
	
Recall that for elementary tensors $\xi,\eta \in \Fock_Y$ we have $\alpha(T^Y_\xi T_\eta^{Y*})  = T_{\xi}^X T_{\eta}^{X*}$.
The previous calculation shows that
$
\Psi_P ( \alpha (T^Y_\xi T_\eta^{Y*})) = T^Y_\xi T_\eta^{Y*},
$
so $\Psi_P \circ \alpha = \Id_{\cT_Y}$. Since $\Psi_P\circ\ol{\iota}=\Psi_P\circ\beta=\Id_{\cT_Y}$ we also have $\Psi_P \circ \beta=\Psi_P \circ \alpha$.

For the first equality of \labelcref{eq:toeplitz-expectation}, fix elementary tensors $y \in Y^{\ox n}_A$ and $\xi \in X^{\ox m}_A$. Then
\begin{align*}
    T^{X*}_{y} T^X_\xi
     & = \begin{cases}
        \langle y \mid \xi_1 \ox \cdots \ox \xi_n \rangle _A  T^X_{\xi_{n+1} \ox \cdots \ox \xi_{m}}      & \text{if } m \ge n \\
        T^{X*}_{{y_{m+1} \ox \cdots \ox y_n}} \langle y_1 \ox \cdots \ox y_{m} \mid \xi\rangle_A & \text{if } n > m
    \end{cases} \\
     & = \begin{cases}
        T^X_{  \phi_X(\langle y \mid P (\xi_1 \ox \cdots \ox \xi_n) \rangle_A) \xi_{n+1} \ox \cdots \ox \xi_{m}}      & \text{if } m \ge n \\
        T^{X*}_{ \phi_X(\langle P\xi \mid y_1 \ox \cdots \ox y_{m}\rangle_A) {y_{m+1} \ox \cdots \ox y_n}} & \text{if } n > m.
    \end{cases}
\end{align*}
So, for elementary tensors $x \in \Fock_Y$ and $\eta \in \Fock_X$,
	\begin{align*}
		\Psi_P(\alpha(T_x^Y T_y^{Y*})T_\xi^X T_\eta^{X*})
		 & = P T_{x}^X T_{y}^{X*} T_\xi^X T_\eta^{X*} P \\
		 & =\begin{cases}
			P T^X_{x \ox \phi_X(\langle y \mid P_0 \xi_1 \ox \cdots \ox P_0\xi_n \rangle_A)   \xi_{n+1} \ox \cdots \ox \xi_{m}} T_\eta^{X*} P     & \text{if } m \ge n \\
			P T_{x}^X 	 T_{\eta \ox \phi_X(\langle P\xi \mid y_1 \ox \cdots \ox y_{m}\rangle_A)  {y_{m+1} \ox \cdots \ox y_n}}^{X*} P & \text{if } n > m
		\end{cases}\\
		& =\begin{cases}
			 T^Y_{x \ox \phi_Y(\langle y \mid P_0 \xi_1 \ox \cdots \ox P_0\xi_n \rangle_A)  P_0 \xi_{n+1} \ox \cdots \ox P_0 \xi_{m}} T_\eta^{Y*}      & \text{if } m \ge n \\
			 T_{x}^Y	 T_{P\eta \ox \phi_Y(\langle P\xi \mid y_1 \ox \cdots \ox y_{m}\rangle_A)  {y_{m+1} \ox \cdots \ox y_n}}^{Y*}  & \text{if } n > m
		\end{cases}         \\
		 & = T_x^Y T_y^{Y*} T_{P\xi}^Y T_{P\eta}^{Y*}                 
		  = T_x^Y T_y^{Y*} \Psi_P(T_{\xi}^X T_{\eta}^{X*}).
	\end{align*}
Similarly, $\Psi_P(T_\xi^X T_\eta^{X*}\alpha(T_x^X T_y^{X*})) =  \Psi_P(T_{\xi}^X T_{\eta}^{X*}) T_x^Y T_y^{Y*}$. 
		
Since $\Psi_P\colon \End_A(\Fock_X) \to \End_A(\Fock_Y)$ is an expectation for the inclusion $\ol{\iota}$, 
it follows that $\Psi_P$ descends to an expectation $\Psi_P \colon \cT_X\to \cT_Y$ for the inclusion $\alpha$. 
The second equality of \labelcref{eq:toeplitz-expectation} follows because $\Psi_P$ is an expectation for $\ol{\iota}$ and $\beta =\ol{\iota}|_{\cT_Y}$.
\end{proof}

With the expectation $\Psi_P \colon \cT_X \to \cT_Y$ in hand, 
we aim to appeal to \Cref{lem:quotient-postitive-map} to induce a strict completely positive map between the corresponding Cuntz--Pimsner algebras.
Since $\cO_X$ is the quotient $\cT_X / \End_A^0(\Fock_X \cdot J_X)$ we need to show that $\Psi_P$ takes $\End_A^0(\Fock_X \cdot J_X)$ to $\End_A^0(\Fock_Y \cdot J_Y)$. 

\begin{lemma}
	\label{lem:fock-projection-compacts}
	Let $\Psi_P \colon \cT_X \to \cT_Y$ be the conditional expectation of \Cref{lem:fock-expectation}. If $J_X = J_Y$, then $\Psi_P(\End_A^0(\Fock_X \cdot J_X)) = \End_A^0(\Fock_Y \cdot J_Y)$.
\end{lemma}

\begin{proof}
	Since $P$ is right $J_X$-linear, we have $\Psi_P(\Theta_{\xi,\eta}) = \Theta_{P\xi,P\eta} \in \End_A^0(\Fock_Y \cdot J_Y)$ for all $\xi,\eta \in \Fock_X \cdot J_X$. Since $P$ is surjective $\Psi_P(\End_A^0(\Fock_X \cdot J_X))=\End_A^0(\Fock_Y \cdot J_Y)$.
\end{proof}

To use \Cref{lem:quotient-postitive-map}, we need the existence of a completely positive splitting $\cO_X \to \cT_X$. 
By the Choi--Effros lifting theorem~\cite[Theorem C.3.]{Brown-Ozawa} such a splitting exists whenever $\cO_X$ is nuclear. 
A sufficient condition for nuclearity of $\cO_X$ is that the coefficient algebra $A$ is nuclear \cite[Corollary 7.4]{KatsuraCPalg}.
The resulting strict completely positive map $\wt{\Psi}_P$ is typically not a conditional expectation itself since there is no obvious choice for an inclusion of $\cO_Y$ as a subalgebra of $\cO_X$.

	 By applying \cref{lem:quotient-postitive-map} the map $\Psi_P \colon \cT_X\to \cT_Y$ passes to the Cuntz--Pimsner quotients.
\begin{thm}
\label{thm:CP-pos-coz}
Let $\cor{\phi_Y}{A}{Y}{A}$ be a complemented sub-$A$--$A$-correspondence of $\cor{\phi_X}{A}{X}{A}$ with $J_X = J_Y$ (for example if $(\phi_{Y^{\perp}}, {}_A Y_A^{\perp})$ is regular), and let $\Psi_P\colon \cT_X\to\cT_Y$ be the conditional expectation of \cref{lem:fock-expectation}. Further suppose that there is a completely positive splitting $\cO_X \to \cT_X$ (for example if  $A$ is nuclear). 
Then there is a strict completely positive map $\wt{\Psi}_P\colon \cO_X\to \cO_Y$.
\end{thm}

\begin{corl}\label{corl:morphisms_from_projection}
	With the setup of \cref{thm:CP-pos-coz} we have the following:
	\begin{enumerate}[label=(\roman*), ref=(\roman*)]
		\item $(\wt{\Psi}_P,{}_{\cO_X}(\cO_Y)_{\cO_Y})$ is a positive correspondence, defining a morphism in $\Quesa$ from $\cO_X$ to $\cO_Y$;
		\item the KSGNS space $(\pi_{\wt{\Psi}_P}, \cO_X \ox_{\wt{\Psi}_P} \cO_Y)$ defines a morphism in $\Ench$ from $\cO_X$ to $\cO_Y$; and
		\item \cref{lem:quotient-expectation} implies that
		$
		(\pi_{\wt{\Psi}_P}, \cO_X \ox_{\wt{\Psi}_P} \cO_Y)\cong (\Id_{\cO_X},L^2_{\cO_Y}(\cO_X,\wt{\Psi}_P)).
		$
	\end{enumerate}
\end{corl}

A direct application of \cref{lem:compose-expectation-quot}  gives the following description of composition. 
\begin{corl} 
\label{cor:sigh}
	Suppose that $(\phi_X, {}_A X_A) \overset{P_0}{\to} (\phi_Y,{}_A Y_A)  \overset{Q_0}{\to} (\phi_Z,{}_A Z_A)$ are correspondence projections
	that induce correspondence projections $(\phi_X, \Fock_X) \overset{P}{\to} (\phi_Y, \Fock_Y) \overset{Q}{\to} (\phi_Z, \Fock_Z)$. 
    Suppose that $J_Z = J_Y = J_Z$ and that there are completely positive splittings $\cO_X \to \cT_X$ and $\cO_Y \to \cT_Y$.  Then there is an isometric inclusion of correspondences
	\[
	\psi \colon (\cO_X \ox_{\wt{\Psi}_P} \cO_Y) \ox ( \cO_Y \ox_{\wt{\Psi}_Q} \cO_Z ) \to \cO_X \ox_{\wt{\Psi}_{QP}} \cO_Z.
	\]	
	Moreover, if there is a completely positive splitting $s_{Y} \colon \cO_Y \to \cT_Y$ with the additional property that 
    the subspace $\alpha \circ s_Y(\cO_Y) + \End_A^{0}(\Fock_X \cdot J_X) $ contains an approximate unit for $\cO_X$, then $\psi$ is also surjective.
\end{corl}
\begin{proof}
    It follows from from \cref{thm:CP-pos-coz} and \cref{lem:compose-expectation-quot} that there is an isometric embedding 
\[
\psi\colon (\cO_X \ox_{\wt{\Psi}_P} \cO_Y) \ox ( \cO_Y \ox_{\wt{\Psi}_Q} \cO_Z ) \to \cO_X \ox_{\wt{\Psi}_{PQ}} \cO_Z.
\]	
If $s\colon \cO_X \to T_X$ is a nondegenerate splitting of the canonical quotient map $\T_X \to \cO_X$ 
(if the \mbox{$C^*$-algebras} are unital, then $s$ can be chosen to be unital), 
then there is an approximate unit $(e_i)_i$ in $\cO_X$ such that $(s(e_i))_i$ converges strictly to the identity in $\Mult(T_Y)$.
Then $(s(e_i))_i$ is an approximate unit for $\T_X$ and it follows from the last part of \cref{lem:compose-expectation-quot} that $\psi$ is an isomorphism.
\end{proof}

\begin{rmk}
  The isometric inclusion $\psi$ of \cref{cor:sigh} is an isomorphism when $\T_Y$ and $\cO_Y$ are unital,
  because the Choi--Effros lifting $s_Y$ can be chosen to be unital,
  and we suspect that $\psi$ is always an isomorphism.
\end{rmk}

Since the KSGNS construction is fairly explicit, it is often possible to get one's hands on the module $\cO_X \ox_{\wt{\Psi}_P} \cO_Y$ in examples. 
One class of examples arises from graph \mbox{$C^*$-algebras}.

\begin{example}
Let $E = (E^0,E^1,r,s)$ be a directed graph.  
We refer to \cite[Chapter 8]{Raeburn} for the description of the \emph{graph $C^*$-algebra} $C^*(E)$ as the Cuntz--Pimsner $\cO_{X(E)}$ of the graph correspondence $X(E)$.

Suppose that $F = (F^0,F^1,r,s)$ and $E \setminus F \coloneqq (E^0, E^1 \setminus F^1,r,s)$  are subgraphs of $E$ such that $E^0 = F^0$. Further, suppose that $E \setminus F$ is a \emph{regular} subgraph in the sense that it contains no sources or infinite receivers. 
 In particular, $E \setminus F$ must be a ``large'' subgraph of $E$ since isolated vertices are sources.
Then the graph correspondence $X(E)$ splits into a direct sum $X(E) = X(F) \oplus X(E \setminus F)$ of correspondences over $C_0(E^0)$.
Since $X(E\setminus F)$ is regular, \Cref{thm:CP-pos-coz} provides a $C^*(E)$--$C^*(F)$-correspondence.

To be more explicit, let $\{S_\mu \colon \mu \in E^*\}$ be the generating partial isometries of $C^*(E)$ and let $\{W_\mu \colon \mu \in F^*\}$ be the generating partial isometries of $C^*(F)$. 
For $\mu,\nu \in E^*$ we have
\[
    \wt{\Psi}_P (S_{\mu}S_{\nu}^*) = 
    \begin{cases}
        W_{\mu}W_{\nu}^* &\text{if } \mu,\nu \in F^*\\
        0 &\text{otherwise}.
    \end{cases}
\]
In the KSGNS module $C^*(E) \ox_{\wt{\Psi}_P} C^*(F)$, the inner product satisfies
\begin{align*}
    \langle S_\mu S_\nu^* \ox W_\xi W_\eta^* \mid S_{\alpha} S_{\beta}^* \ox W_{\rho} W_{\sigma}^* \rangle_{C^*(F)}
    &= W_\eta W_\xi^* \wt{\Psi}_P(S_\nu S_\mu^* S_\alpha S_\beta^*) W_{\rho} W_{\sigma}^*.
\end{align*}

In particular,
    if $\mu = \alpha$, $\nu = \beta$, $\xi = \rho$, and $\eta = \sigma$, then
\begin{align*}
\langle S_\mu S_\nu^* \ox W_\xi W_\eta^* \mid S_{\mu} S_{\nu}^* \ox  W_\xi W_\eta^*\rangle_{C^*(F)} 
    &= W_\eta W_\xi^* \wt{\Psi}_P(S_\nu S_\mu^* S_\mu S_\nu^*) W_{\xi} W_{\eta}^* \\
    &= \begin{cases}
      W_\eta W_\xi^* W_{\nu} W_{\nu}^* W_\xi W_\eta^* & \text{if } \nu \in F^*         \\
      0 & \text{otherwise}
    \end{cases} \\
    &=
    \begin{cases}
        W_{\eta\nu'} W_{\eta \nu'}^* & \text{if } \nu = \xi \nu' \text{ and } \nu \in F^* \\
        W_{\eta}W_{\eta}^* & \text{if } \xi = \nu \xi' \text{ and } \nu \in F^*\\
        0 & \text{otherwise}.
    \end{cases}
\end{align*}	
Therefore, a necessary condition for $S_\mu S_\nu^* \ox W_\xi W_\eta^* \in C^*(E) \ox_{\wt{\Psi}_P} C^*(F)$ to be nonzero is  
that $\nu \in F^*$ and that $\xi$ extends $\nu$ (or vice versa).

We claim that $C^*(E) \ox_{\wt{\Psi}_P} C^*(F)$ is isomorphic to $\Fock_{X(E)}\ox_{C_0(E^0)} C^*(F)$ as right $C^*(F)$-modules. Recall that $E^k$ denotes the paths of length $k$ in $E$, and each $k \ge 0$ we can identify $X(E)^{\ox k}$ with a completion of $C_c(E^k)$. With this identification, for each $\mu \in E^*$ we let $\delta_\mu \in \Fock_{X(E)}$ denote the point mass at $\mu$ in the completion. 
In $\Fock_{X(E)}\ox_{C_0(E^0)} C^*(F)$ we calculate the inner product 
\begin{align*}
  \langle \delta_\mu \ox W_\nu^*W_\xi W_\eta^* \mid \delta_\mu \ox W_\nu^*W_\xi W_\eta^* \rangle_{C^*(F)}
  &= \langle W_\nu^*W_\xi W_\eta^* \mid P_{s(\mu)} W_\nu^*W_\xi W_\eta^* \rangle_{C^*(F)} \\
  &= \begin{cases}
  	W_\eta W^*_\xi W_\nu W_\nu^*W_\xi W_\eta^* & \text{if } s(\mu) = s(\nu)\\
  	0 & \text{otherwise}.
  \end{cases}
\end{align*}
This means that there is an isometric linear map $\kappa\colon C^*(E)\ox_{\tilde{\Psi}} C^*(F) \to \Fock_{X(E)}\ox_{C_0(F^0)} C^*(F)$ given by 
$\kappa(S_\mu S_\nu \ox W_\xi W_\eta^*) = \delta_\mu \ox W_\nu^* W_\xi W_\eta^*$, for every $\mu\in E^*$ and $\nu, \xi, \eta\in F^*$.
The target $\Fock_{X(E)}\ox_{C_0(F^0)} C^*(F)$ is generated by elements of the form $\delta_\mu \ox W_\xi W_\eta^*$, and this is nonzero only if $r(\xi) = s(\mu)$.
Choosing $\nu = s(\mu)$ so that $S_\nu = P_{s(\mu)}$, we see that $\kappa(S_\mu S_\nu \ox W_\xi W_\eta^*) = \delta_\mu \ox W_\xi W_\eta^*$.
This shows that $\kappa$ is surjective. 

The map $\kappa$ clearly preserves the right action of $C^*(F)$ and a routine computation shows that $\kappa$ preserves inner products. 
Define $\rho \colon C^*(E) \to \End_{C^*(F)} (\Fock_{X(E)}\ox_{C_0(F^0)} C^*(F))$ by
\[
\rho(S_{\alpha}S_{\beta}^*) ( \delta_\mu \ox W_\nu^*W_\xi W_\eta^*)
= \begin{cases}
  	\delta_{\alpha \mu'} \ox W_\nu^*W_\xi W_\eta^* & \text{if }  \mu = \beta \mu'\\
  	\delta_{\alpha} \ox W_{\nu \beta'}^* W_{\xi} W_{\eta}^* & \text{if } \beta = \mu \beta' \text{ and } \beta' \in F^* \\
	0 & \text{otherwise}.
\end{cases}
\]
Comparing with the (positive) left action on the KSGNS module completes the isomorphism of correspondences.

For a particular instance, suppose that $E^{0} = \{v\}$ and that there are $n$ edges (which must have source and range $v$). 
Let $F$ be the subgraph consisting of $m \le n$ edges. 
Then $E \setminus F$ is regular, $X(E) \cong \bC^n$, and $X(F) \cong \bC^n$.
The graph algebra $C^*(E)$ is isomorphic to the Cuntz algebra $\cO_n$, and $C^*(F)\cong \cO_m$. 
Our construction provides a positive $\cO_n$--$\cO_m$-correspondence with underlying $\cO_m$-module isomorphic to $\Fock_{\bC^n} \ox_{\bC} \cO_m$.
Note that $\cO_n$ typically does not act by adjointable operators on $\Fock_{\bC^n}$, but it does act on the module $\Fock_{\bC^n} \ox_{\bC} \cO_m$. 
\end{example}

\subsection{Fock projections and subproduct systems}
\label{sub:Fock-proj}

The previous subsection started from a correspondence projection, then built a correspondence projection on the associated Fock module, and finally constructed a positive map.
In certain circumstances we can start directly with a projection 
on the Fock module and obtain sensible morphisms 
on the associated Cuntz--Pimsner algebras.

\begin{defn}
Let $(\phi,{}_AX_A)$ be a correspondence, and let 
$Q_n \colon \Fock_X\to X^{\ox n}$ be the projection onto the $n$-th summand. 
A \emph{Fock projection} $P\in\End_A(\Fock_X)$ is
a projection which commutes with the left action of $A$ on $\Fock_X$, and $Q_n$ for all $n \ge 0$. We write $P_n\coloneqq Q_nP$ and note that the projection $P$ on $\Fock_X$ is equal to the strict sum $\sum_{n=0}^\infty P_n$.
\end{defn}

\begin{example}
A correspondence projection $P_0\colon (\phi_X,{}_AX_A)\to (\phi_Y,{}_AY_A)$ induces a Fock projection as in \cref{lem:fock-projection}.
\end{example}

\begin{lemma}
Let $(\phi,{}_A X_A)$ be a correspondence and let $P\in\End_A(\Fock_X)$
be a Fock projection. The map $\Psi_P \colon \End_A(\Fock_X) \to \End_A(P\Fock_X)$ given by $\Psi_P(T) = PTP$ is a conditional expectation and restricts to
a surjective completely positive map from
$\End_A^0(\Fock_X \cdot J_X)$ to $\End_A^0(P\Fock_X \cdot J_X)$.
\end{lemma}
\begin{proof}
  This result follows from the proof of \cref{lem:fock-expectation,lem:fock-projection-compacts} with minor modification.
\end{proof}

If we restrict $\Psi_P$ to a subalgebra of $\End_A(\Fock_X)$, such as $\cT_X$, we get a 
completely positive map onto \emph{some} subalgebra of 
$\End_A(P\Fock_X)$. 
\begin{defn}
Let $(\phi,{}_A X_A)$ be a correspondence, and let $P\in\End_A(\Fock_X)$
be a Fock projection. For $\xi\in P\Fock_X$ we define a \emph{projected creation operator} $T^P_\xi\coloneqq PT_\xi$, where $T_\xi \in \cT_X$. Then $T^P_\xi$ is adjointable with adjoint $(T^P_\xi)^*=T^*_\xi P$. 

We define the \emph{projected Toeplitz algebra} $\cT^P_X$ of $P\Fock_X$ to be the \mbox{$C^*$-algebra} generated by $\{T^P_\xi \mid \xi\in P\Fock_X\}$. We define the \emph{projected Cuntz--Pimsner algebra} $\cO^P_X$ of $P\Fock_X$ to be the quotient of $\cT^P_X$ by the ideal
\[
\End^0_A(P\Fock_X\cdot J_X)\cap \cT^P_X.
\]
\end{defn}

By construction, $\Psi_P \colon \cT_X \to \cT_X^P$ is a conditional expectation. 
Building on similar ideas to the previous section, we use \cref{lem:quotient-postitive-map} to construct a strict completely positive map between a Cuntz--Pimsner algebra and its projected relative. 

By applying \cref{lem:quotient-postitive-map} the map $\Psi_P\colon \cT_X\to\cT_X^P$ passes to the Cuntz--Pimsner quotients. 
\begin{prop}
\label{prop:Fock-on}
Let $(\phi,{}_A X_A)$ be a \mbox{$C^*$-correspondence} such that there is a completely positive splitting $\cO_X \to \cT_X$ (for example if $A$ is nuclear). Let $P\in\End_A(\Fock_X)$
be a Fock projection. Then there is a surjective strict completely positive map
$
\wt{\Psi}_P:\cO_X\to \cO_X^P.
$
\end{prop}

The construction of $\cO_X^P$ is closely related to the construction of subproduct systems over $\N$. 
The various definitions found in \cite{SS09,V,DM14,AK} (for instance) can all be seen to be equivalent using \cite[Lemma 6.1]{SS09}.

\begin{defn}
\label{defn:subby}
Let $A$ be a separable and nuclear \mbox{$C^*$-algebra,}  and let $(X_m)_{m\in\N_0}$ be a sequence of nondegenerate  $A$--$A$-correspondences with $X_0=A$. We say that $(X_m)_{m\in\N_0}$ is a \emph{subproduct system} if for all $n,m\in\N_0$ the correspondence $X_{n+m}$ is a complemented sub-correspondence of $X_n\ox_AX_m$.
\end{defn}

One can now deduce the existence of projections $P_k\colon X_1^{\ox_Ak}\to X_k$, cf. \cite[Lemma 6.1]{SS09}.
Setting $P_0={\rm Id}_A$, the strict sum 
\begin{equation}
P\coloneqq\sum_{k\in \N_0}P_k
\label{eq:so-Fock}
\end{equation} 
is a Fock projection on $\Fock_{X_1}$.

Given a subproduct system we can construct the associated Toeplitz and Cuntz--Pimsner algebras. We use the definition of \cite{AK}, which is derived from \cite{V}.

\begin{defn}
Let $X=(X_m)_{m\in\N_0}$ be a subproduct system over a separable and nuclear \mbox{$C^*$-algebra} $A$, and let $P$ be the Fock projection \labelcref{eq:so-Fock}. 
The \emph{Toeplitz algebra} of $X$ is defined to be $\cT^P_{X_1}$. 
Setting $\mathbb{I}=\{T\in\cT^P_X:\,\lim_{m\to\infty}\Vert Q_m T\Vert=0\}$, the \emph{Cuntz--Pimsner} of $X$ is
$\mathbb{O}_X\coloneqq\cT_{X_1}^P/\mathbb{I}$.
\end{defn}

Given $X=(X_m)_m$, the distinction between $\cO_{X_1}^P$ and $\mathbb{O}_X$ is the distinction between the ideals 
$\End^0_A(P \Fock_{X_1}\cdot J_{X_1})$ and $\mathbb{I}$, see
\cite{V}.

\begin{prop}
\label{prop:Fock-subby}
Let $X=(X_m)_{m\in\N_0}$ be a subproduct system over a separable and nuclear \mbox{$C^*$-algebra} $A$, and let $P$ be the Fock projection \labelcref{eq:so-Fock}. 
Then there is a strict completely positive map
$
\Phi_P \colon \cO_{X_1}\to \mathbb{O}_X. 
$
\end{prop}

\begin{proof}
We have an expectation $\Psi_P \colon \cT_{X_1} \to \cT_{X_1}^P$ which, by  \cref{prop:Fock-on}, descends to a strict completely positive map $\wt{\Psi}_P\colon \cO_X\to \cO_X^P$. Since $\End^0_A(P\Fock_{X_1}\cdot J_{X_1})P\subseteq \cT^P_{X_1}$ and as 
$\Vert Q_mT\Vert\to 0$  for every $T\in \End^0_A(P\Fock_{X_1}\cdot J_{X_1})P$, we have $\End^0_A(P\Fock_{X_1}\cdot J_{X_1})\subset\mathbb{I}$. Consequently, the map $\lambda \colon \cO_X^P \to \mathbb{O}_X$ defined by
$\lambda (T + \End^0_A(P\Fock_{X_1}\cdot J_{X_1})) = T + \mathbb{I}$ is a well-defined surjective $*$-homomorphism. Then $\Phi_P \coloneqq \lambda \circ \wt{\Psi}_P$ is the desired strict completely positive map. 
\end{proof}

\begin{corl}
Let $X=(X_m)_{m\in\N_0}$ be a subproduct system over a separable nuclear \mbox{$C^*$-algebra} $A$, and let $P$ be the Fock projection \labelcref{eq:so-Fock}. Then 
	\begin{enumerate}[label=(\roman*), ref=(\roman*)]
	\item $({\Phi}_P,{}_{\cO_{X_1}}(\mathbb{O}_X)_{\mathbb{O}_X})$ is a positive correspondence, defining a morphism in $\Quesa$ from $\cO_{X_1}$ to $\mathbb{O}_X$, and
	\item the KSGNS space $(\pi_{\Phi_P},\cO_{X_1} \ox_{\Phi_P} \mathbb{O}_X)$ defines a morphism in $\Ench$ from $\cO_{X_1}$ to $\mathbb{O}_X$.
\end{enumerate}
\end{corl}

\section{Morphisms from bi-Hilbertian bimodules}
\label{sec:bi-curious}

In this section we change tack and consider another class of examples of positive correspondence which arise by trying to generalise the covariant correspondences of \cite{Meyer-Sehnem}.

\subsection{Bi-Hilbertian bimodules}
\label{subsec:biHilb-bimod}

Amongst nondegenerate \mbox{$C^*$-correspondences} there is a useful subclass more general than Morita equivalences, but much tamer than general \mbox{$C^*$-correspondences}. 
The properties of these {\em bi-Hilbertian bimodules} were developed by Watatani and coauthors (we refer to \cite{KajPinWat} for most facts) 
as a $C^*$-algebraic version of the setting of Jones index theory, 
and have since arisen in work of \cite{CCH} on representation theory, \cite{RRSext,RRSPD,GMR,AR} on the Kasparov theory of Cuntz--Pimsner algebras, 
and in the characterisation of ``non-commutative vector bundles'' \cite{RSims}.

In this section we assume that all $C^*$-modules are countably generated.

\begin{defn}
	\label{defn:bimod} Let $A$ and $B$ be \mbox{$C^*$-algebras}. Following \cite{KajPinWat}, a
	\emph{bi-Hilbertian $A$--$B$-bimodule} $F$ is a right $B$-module with inner product
	$\langle\cdot\mid\cdot\rangle_B$ which is also a left $A$-module with inner product
	${}_A\langle\cdot\mid\cdot\rangle$ such that
	\begin{enumerate}
		\item the left action of $A$ on $F$ is adjointable with respect to
		$\langle\cdot\mid\cdot\rangle_B$;
		\item  the right action of $B$ on $F$ is adjointable with respect to
		${}_A\langle\cdot\mid\cdot\rangle$; and
		\item the left and right inner products induce equivalent norms on $F$.
	\end{enumerate} 
\end{defn}
We usually do not refer to the left and right action $*$-homomorphisms and simply write ${}_A F_B$ for a bi-Hilbertian $A$--$B$-bimodule since the definition is symmetric, 
unlike that of a \mbox{$C^*$-correspondence}.
By \cite[Proposition~2.16]{KajPinWat} the actions on ${}_A F_B$ are automatically nondegenerate. 
Many examples can be found in \cite[Section 2.1]{RRSext}.

Before we examine bi-Hilbertian bimodules,
we address the question of when a correspondence has a compatible left inner product. 
 A less general version of the following lemma appears in \cite[Lemma 3.23]{LRV}. 

\begin{lemma}[cf. {\cite[Lemma~2.6]{KajPinWat}}]\label{lem:Watatani+}
	Let $F_B$ be a countably generated right $B$-module, and let $A\subseteq \End_B(F)$ be a $C^*$-subalgebra.
	Suppose that ${}_A\langle\cdot\mid\cdot\rangle$ is a left $A$-valued
	inner product on $F_B$ for which the right action of $B$ is adjointable. Then there is a
	$A$-bilinear faithful positive map $\Upsilon \colon\End_B^{00}(F) \to A$ such that
	$\Upsilon(\Theta_{e,f}) \coloneqq {}_A\langle e\mid f\rangle$ for all $e,f \in F_B$. For any frame $(f_i)$
	for $F_B$, we have
	\[
		\Upsilon(T) = \sum_i {}_A\langle T f_i\mid f_i\rangle\quad{\rm for\ all }\ T \in \End_B^{00}(F).
	\]
	On the other hand if there is an $A$-bilinear faithful positive map $\Upsilon \colon\End_B^0(F) \to A$, then ${}_A\langle e\mid f\rangle \coloneqq \Upsilon(\Theta_{e,f})$ defines a left
	$A$-valued inner product on $F_B$ for which the right action of $B$ is adjointable.
\end{lemma}

\begin{defn}[{\cite[Definition~2.8]{KajPinWat}}]
	\label{defn:fin-wat}
	Let ${}_A F_B$ be a bi-Hilbertian $A$--$B$-bimodule. We say that ${}_A F_B$ has \emph{finite
		right numerical index} if there exists $\lambda
		> 0$ such that
	\begin{equation}
      \big\|\sum_{i=1}^n {}_A \langle f_i \mid f_i\rangle\big\| \le \lambda \big\|\sum_{i=1}^n \Theta_{f_i, f_i}\big\|
		\quad\text{ for all $n\in\N$ and all $f_1, \dots , f_n \in F$.}
	\label{eq:num-ind}
	\end{equation}
	The \emph{right numerical index} of ${}_A F_B$ is the infimum of the numbers $\lambda$ satisfying
	\labelcref{eq:num-ind}.
\end{defn}
\begin{rmk}
	Following~\cite[Corollary~2.11]{KajPinWat}, if ${}_A F_B$ has finite right numerical index, then the map $\Upsilon \colon\End_B^{00}(F) \to A$ of \cref{lem:Watatani+} 
    extends to an $A$-bilinear $*$-homomorphism $\Upsilon \colon\End_B^{0}(F) \to A$ satisfying
    \[
    \Upsilon(\Theta_{e,f}) = {}_A \langle e \mid f \rangle \quad \text{for all } e \in E,\, f \in F.
    \]
	If $T\in \End_B^{0}(F)$ commutes with all $a\in A$, then
	$
	a\Upsilon(T)=\Upsilon(aT)=\Upsilon(Ta)=\Upsilon (T)a,
	$
	so $\Upsilon(T)$ is in the centre of $A$.	
	In particular, $\Upsilon({\rm Id}_F)$ is central in $A$.
\end{rmk}
 In \cite[Proposition~2.13]{KajPinWat} it is shown that tensor products of bi-Hilbertian bimodules with finite left and right numerical indices are again bi-Hilbertian bimodules.

Let ${}_A F_B$ be a bi-Hilbertian bimodule, countably generated as a right module and with
finite right numerical index. By \cite[Theorem 2.22]{KajPinWat}, the left
action of $A$ on ${}_A F_B$ is by compact endomorphisms with respect to $\langle\cdot \mid \cdot\rangle_B$
if and only if, for every frame $(f_j)$ for the right $B$-module $F_B$, the series
\begin{equation}
\label{eq:rInd}
    \sum_{j\geq 1}{}_A\langle f_j \mid f_j\rangle
\end{equation}
converges strictly in $\Mult(A)$. In this case we denote the strict limit by $\rInd(F)$.
We note that $\rInd(F)$ is independent of the frame $(f_j)$, \cite[Theorem 2.22]{KajPinWat}.

\begin{defn} When it exists we call $\rInd(F)$ the
\emph{right Watatani index} of ${}_A F_B$. 
\end{defn}

\begin{defn}
	We say that ${}_A F_B$ has \emph{finite right Watatani index} if
	it has finite right numerical index and the left action is by compacts. The left numerical index and left index are defined similarly. If ${}_A F_B$ has finite left and right index, we say that 
	${}_A F_B$ has \emph{finite index}.
\end{defn}
\begin{example}[{\cite[Corollary 4.14]{KajPinWat}}]
	A Morita equivalence $A$--$B$-bimodule is a finite index bi-Hilbertian $A$--$B$-bimodule with
	right index $1 \in \Mult(A)$ and left index $1 \in \Mult(B)$. On the other hand, if a finite index bi-Hilbertian $A$--$B$-bimodule ${}_A F_B$ has right index $1 \in \Mult(A)$ and left index $1 \in \Mult(B)$, then $F_B$ can be equipped with a left inner product making $F_B$ into a Morita equivalence $A$--$B$-bimodule. 
\end{example}

We record the following facts about $\rInd(F)$.
\begin{lemma}[{\cite[Corollary 2.28]{KajPinWat}}]\label{lem:nice_index}
	Let ${}_A X_B$ be a bi-Hilbertian bimodule with finite right Watatani index. Then $\rInd(F)$ is a positive central element of $\Mult(A)$. Moreover, the following are equivalent:
	\begin{enumerate}
		\item $\rInd(F)$ is invertible,
		\item the left action $A \to \End_A^0(F)$ is injective, and
		\item the left inner product ${}_A\langle \cdot \mid \cdot \rangle$ is full.
	\end{enumerate}
\end{lemma}
\cref{lem:nice_index} motivates the following definition.
\begin{defn}
	A bi-Hilbertian bimodule is \emph{regular} if it has finite index and both the left and right inner products are full (equivalently the left and right actions are injective).
\end{defn}

If ${}_AF_B$ is a regular bi-Hilbertian bimodule then the left and right Watatani indices are central, positive and invertible elements of $\Mult(A)$ and $\Mult(B)$, respectively. 
This allows us to write $e^{\beta} \in \Mult(A)$ for the right Watatani index, where $\beta = \beta_R$ is central and self-adjoint. 
The notation $e^{\beta}$ will always mean the \emph{right} Watatani index unless stated otherwise. 
When needed, the left Watatani index is denoted $e^{\beta_L}$.

The next result gives useful relationships between frames for the different module structures on a bi-Hilbertian bimodule, its conjugate module, and its algebra of compact operators. 
Recall from \cref{ntn:compacts} that if ${}_A F_B$ is a bi-Hilbertian bimodule then $\End^0_B(F)$ is identified with $F \ox_B F^*$. 
Since ${}_A F_B$ comes equipped with a left $A$-valued inner product, $\End_B^0(F)$ inherits the structure of a bi-Hilbertian $A$-bimodule, {\cite[Corollary~2.29]{KajPinWat}}.

\begin{lemma}
	Let ${}_A F_B$ be a bi-Hilbertian bimodule. 
	Given a right frame $(u_j)_j\subseteq F_B$ and left frame $(\wt{u}_k)_k\subseteq {}_AF$, a right frame for ${}_A F\ox_BF^*_A$ is given by $(u_j \ox \wt{u}_k^*)_{j,k}$, 
    and a left frame is given by $(\wt{u}_j \ox u_k^*)_{j,k}$.
	If ${}_A F_B$ has finite index then so does ${}_A F\ox_BF^*_A$, and if ${}_A F_B$ is regular then so is ${}_A F\ox_BF^*_A$.
	Analogous statements also hold for ${}_B F^* \ox_A F_B$. 
\end{lemma}
\begin{proof}
	The statement about frames follows from \cite[Proposition~2.16]{Grooves}.	
	If ${}_A F_B$ has finite index, then so does ${}_A F \ox_B F_A^*$, by \cite[Theorem 5.1]{KajPinWat}. 
    If the left action $\phi \colon A \to \End^0_B(F)$ is injective, then so is $\phi \ox \Id_{F^*}$ by \cite[Lemma 4.7]{KatsuraCPalg}. 
    Similarly, if the right action on ${}_A F_B$ is injective, then the right action on ${}_A F \ox_B F_A^*$ is injective.
\end{proof}

The following property was discovered and described in \cite[Section 4.2]{KajPinWat} using the language of intertwiners.

\begin{lemma}
	\label{lem:complemented_bimodule}
	Let ${}_A F_B$ be a regular bi-Hilbertian bimodule. Then $P_0\colon F\ox_B F^* \to F\ox_B F^*$ defined by 
	\[
	P_0 (T) = e^{-\beta}\Upsilon(T)\, {\rm Id}_F 
	\]
	is an $A$-bilinear adjointable projection. In particular, $P(F \ox_B F^*) = A  \Id_F$ is a complemented bi-Hilbertian sub-bimodule of ${}_A F \ox_B F^*_B$, which is isomorphic to ${}_A A_A$.
\end{lemma}

\begin{proof}
Recall that if $e^{\beta} \in \Mult(A)$ has finite index then $\Upsilon(T) \in A \subseteq F \ox_B F^*$ for all $T\in F\ox_B F^*$. Since $F\ox_B F^*$ is an ideal in $\End_B(F)$, we have $P_0(T) \in F \ox_B F^*$. 	
    The $A$-bilinearity of $\Upsilon$ and centrality of $e^{\beta}$ imply that $P_0$ is $A$-bilinear. Fullness of the left inner product implies that $P_0$ surjects onto $A \Id_F$.
    
   	To see that $P_0$ is idempotent, fix a right frame $(u_i)_i$ for $F_B$. Then for $a \in A$ we have 
	\[
	P_0(a \Id_F) = \lim_i \sum_i P_0(a u_i \ox u_i^*) = a \Id_F. 
	\]
	Since $e^{-\beta}\Psi(T)\in A$ it follows that $P^2(T) = P_0(e^{-\beta} \Upsilon(T) \Id_F) = e^{-\beta} \Upsilon(T) \Id_F$.
	
	For adjointability of $P_0$, fix $e_1 \ox e_2^*,\,f_1 \ox f_2^* \in F \ox_B F^*$. We first observe that
	\[
		\sum_l\langle u_l\ox u_l^*\mid f_1\ox f_2^*\rangle_A
		=\sum_l\langle u_l^*\mid\langle u_l\mid f_1\rangle_A \cdot f_2^*\rangle_A
		=\sum_l\big\langle (\langle f_1\mid u_l\rangle_A \cdot u_l)^*\mid f_2^*\big\rangle_A
		={}_A\langle f_1\mid f_2\rangle.
    \]
	Then, using centrality of $e^{-\beta}$ at the last equality,
	\begin{align*}
		\langle P_0(e_1 \ox e_2^*) \mid f_1 \ox f_2^* \rangle_A
		&=
		\sum_k\langle e^{-\beta} {}_A \langle e_1 \mid e_2 \rangle u_k\ox u_k^*\mid f_1\ox f_2^*\rangle_A\\
		&=\sum_k\langle u_k\ox u_k^*|e^{-\beta} {}_A \langle e_2 \mid e_1 \rangle f_1\ox f_2^*\rangle_A
		={}_A\langle e^{-\beta}{}_A \langle e_2 \mid e_1 \rangle f_1\mid f_2\rangle  \\                    
		 &=e^{-\beta}{}_A \langle e_2 \mid e_1 \rangle  {}_A\langle f_1\mid f_2\rangle
		  ={}_A \langle e_2 \mid e_1 \rangle e^{-\beta}  {}_A\langle f_1\mid f_2\rangle.
	\end{align*}
	A symmetric calculation shows $\langle e_1 \ox e_2^* \mid P_0(f_1 \ox f_2^*) \rangle_A  ={}_A \langle e_2 \mid e_1 \rangle e^{-\beta}  {}_A\langle f_1\mid f_2\rangle$, from which it follows that $P_0^* = P_0$.
\end{proof}

A bi-Hilbertian bimodule ${}_AF_B$  does not typically give a covariant correspondence between a \mbox{$C^*$-correspondence} ${}_AX_A$ and $F^*\ox_AX\ox_AF$ in the sense of \cite[Definition 2.21]{Meyer-Sehnem}. 
This is because $X\ox_A F \cong A \ox_A X \ox_A F \subseteq F\ox_B F^* \ox_A X \ox_A F$,
where the inclusion is as a complemented submodule. 
The only time this inclusion is an isomorphism is when $F$ is a Morita equivalence. 

Although they do not typically induce covariant correspondences, in the next section we show that conjugation of a \mbox{$C^*$-correspondence} by a bi-Hilbertian bimodule 
yields a strict completely positive map between corresponding Cuntz--Pimsner algebras. 

\subsection{From bi-Hilbertian bimodules to positive maps}
\label{subsec:better-name?}

Suppose that ${}_{A} F_B$ is a bi-Hilbertian bimodule of finite index, 
and ${}_AX_A$ is a \mbox{$C^*$-correspondence}.
\cref{lem:complemented_bimodule} provides natural projections on
the Fock module of the $B$--$B$-correspondence $F^*XF \coloneqq {}_B(F^*\ox_AX\ox_AF)_B$. We use these projections to induce
correspondences between their associated Cuntz--Pimsner algebras.   We start on Fock space and build expectations on the associated Toeplitz algebras.

Let $P_0 \colon F \ox_B F^* \to A  \Id_F$ be the projection of \cref{lem:complemented_bimodule}, so 
\[
P_0(f \ox g^*) = e^{-\beta} {}_A\langle f \mid g\rangle   \Id_F,
\]
where $e^{\beta}$ is the right Watatani index of $F_B$.
Consider the element
\begin{equation}
	Z\coloneqq e^{- \beta/2}  \Id_F \in \End_B(F).
	\label{eq:Zeds-dead}
\end{equation}
Since $F$ is finite index, $A$ acts compactly and so $aZ\in \End_B^0(F)=F\ox_B F^*$ for all $a\in A$.
Then for each $a \in A$, and right frame $(u_i)$ we have
\[
P_0(aZ) = e^{-\beta/2}\sum_i  P_0(a u_i \ox u_i^*)  = e^{-\beta/2} \sum_i e^{-\beta} a{}_A \langle u_i \mid u_i \rangle  \Id_F = aZ.
\] 
Since $e^{\beta} \in \Mult(A)$ commutes with elements of $A$, $Z$  commutes with elements of $A \subseteq \End_B^0(F)$. We record how inner products with $aZ$ work in $\End_B^0(F)$.
\begin{lemma}
	\label{lem:Z_facts}
	Let ${}_A F_B$ be a regular bi-Hilbertian bimodule. Then for $f_1 \ox f_2^* \in F \ox_B F^*$ and $a \in A$, we have 
    \[
      \langle f_1 \ox f_2^* \mid  aZ \rangle_A =  e^{-\beta/2} {}_A \langle f_1 \mid f_2 \rangle a.
	\]
	 In particular, for $a_1,\,a_2 \in A$, we have $\langle a_1 Z \mid a_2 Z \rangle_A = a_1^* a_2$.
\end{lemma}
\begin{proof}
	We compute,
	\begin{align*}
		\langle  f_1 \ox f_2^* \mid aZ \rangle_{A}
		&= \sum_i \langle f_2^* \mid \langle f_1 \mid u_i \rangle_B \cdot u_i^* \rangle_A e^{-\beta/2}a
    	= \sum_i\langle (f_2 \cdot \langle f_1 \mid u_i \rangle_B)^* \mid u_i^* \rangle_A e^{-\beta/2}a\\
    	&= \sum_i {}_A \langle f_2 \mid u_i \cdot \langle u_i \mid f_1 \rangle_B \rangle e^{-\beta/2}a
    	= e^{-\beta/2}{}_A \langle f_2 \mid f_1 \rangle a.
	\end{align*}
	The second statement follows from the first. 
\end{proof}
	Let ${}_A F_B$ be a regular bi-Hilbertian bimodule and let $(\phi,{}_A X_A)$ be a $C^*$-correspondence.
 Since $e^{-\beta}$ is positive and invertible we have $P_0(F\ox F^*) = A Z$.  
Using $A$-bilinearity of $P_0$ we can define a projection $P_n$ on $F\ox_B(F^*XF)^{\ox_A n}\ox_BF^*$ by
\[
P_n\coloneqq P_0\ox{\rm Id}_X\ox P_0\ox{\rm Id}_X\ox P_0\ox
\cdots \ox P_0\ox {\rm Id}_X\ox P_0,
\]
where $P_0$ appears $n+1$ times. 
Since the image of $P_0$ is isomorphic to ${}_A A_A$, the image of $P_n$ can be identified with $X^{\ox n}$. We make this more precise.

\begin{lemma}
	\label{lem:fock-complement}
	Let $(\phi,{}_A X_A)$ be a \mbox{$C^*$-correspondence} and let ${}_A F_B$ a regular bi-Hilbertian
	bimodule. For all $n \in \N_0$, there is an adjointable $A$-bilinear isometric inclusion 
	$
	W_n \colon X^{\ox_A n} \to  F\ox_B(F^*XF)^{\ox_A n}\ox_BF^*
	$
	such that for $n \ge 1$,
	\[
	W_n(x_1 \ox \cdots \ox x_n) =  Z \ox x_1 \ox Z \ox \cdots \ox Z \ox x_n \ox Z
	\]
	and $W_0(a) = aZ$ for $a \in A \cong X^{\ox 0}$. For $f,g \in F$ and $\bigotimes_{i=1}^n (f_i \ox x_i \ox g_i^*) \in \Fock_{F^*XF}$, the adjoint satisfies
	\begin{align*}
		W_n^*\Big(f \ox \bigotimes_{i=1}^n (f_i^* \ox x_i \ox g_i) \ox g^* \Big)
		&= e^{-\beta/2} {}_A\langle f \mid f_1 \rangle \cdot x_1 \ox e^{-\beta/2}\cdot {}_A \langle g_1 \ox f_2^* \rangle \cdot x_2 \ox \cdots\\
		&\qquad  \ox e^{-\beta/2} {}_A \langle g_{n-1} \mid f_n \rangle \cdot x_n \cdot e^{-\beta/2} {}_A \langle g_n \mid g \rangle 
	\end{align*}
	Moreover, $
	W_n(X^{\ox_A n}) = P_n(F\ox_B(F^*XF)^{\ox_A n}\ox_BF^*)
	$
	is a complemented sub-correspondence of $F\ox_B(F^*XF)^{\ox_A n}\ox_BF^*$. If ${}_AF_B$ is a Morita equivalence bimodule, then conjugation by ${}_A F_B$ yields
	$X^{\ox_A n}\cong F\ox_B(F^* X F)^{\ox_A n}\ox_B F^*$.
\end{lemma}
\begin{proof}
	For each $x \in {}_A X_A$ we can write $x = a_1 \cdot x' \cdot a_2$ for some $a_1,a_2 \in A$ and $x' \in {}_A X_A$. 
    So $Z \ox x \ox Z = Za_1 \ox x' \ox a_2 Z$ belongs to $F \ox_B F^* \ox_A X \ox_A F \ox_B F^*$. 
    In particular, the formula defining $W_n$ makes sense.	
	
	The case where $n = 0$ follows from the fact that $P_0(F\ox F^*)$ is the $A$-span of $Z$.  
	For $n \ge 1$ we begin by showing that $W_n$ preserves inner products, and is therefore isometric.
	Let $(u_i)$ be a right frame for $F$ and observe that for any $f \in F$ we have $\sum_{i} \langle f\mid u_i\rangle_B \cdot  u_i^* = f^*$. Fix elementary tensors 
	$x = x_1 \ox \cdots  \ox x_n$ and 
	$y = y_1 \ox \cdots \ox y_n$ in $X^{\ox n}$. Then
	\begin{align*}
		\langle Wx \mid Wy\rangle_{A} 
		&= \langle x_n \ox Z \mid \langle W_{n-1}x' \mid W_{n-1} y' \rangle_A \cdot  y_n \ox Z\rangle_A,
	\end{align*}
	where $x' = x_1 \ox \cdots \ox x_{n-1}$ and $y' = y_1 \ox \cdots \ox y_{n-1}$. 
    Using the frame relation at the fourth equality and the definition of $e^{\beta}$ at the fifth equality, 
	\begin{align*}
		\langle W_nx \mid W_ny\rangle_{A} 
		&= \sum_{i,j} \big\langle x_n \ox e^{-\beta/2} \cdot u_i \ox u_i^* \mid  \langle W_{n-1}x' \mid W_{n-1} y' \rangle_A \cdot  y_n \ox e^{-\beta/2} \cdot u_j \ox u_j^*\big\rangle_A\\
		&= \sum_{i,j} \big\langle u_i^* \mid e^{-\beta/2} \cdot u_i \mid \langle x_n \mid \langle W_{n-1}x' \mid W_{n-1} y' \rangle_A \cdot  y_n \rangle_A  e^{-\beta/2} \cdot u_j\rangle_B \cdot u_j^* \big\rangle_A\\
		&= \sum_{i,j} {}_A\big\langle u_i \cdot \langle e^{-\beta/2} \cdot u_i \mid \langle x_n \mid \langle W_{n-1}x' \mid W_{n-1} y' \rangle_A \cdot  y_n \rangle_A  e^{-\beta/2} \cdot u_j\rangle_B \mid  u_j \big\rangle\\
		&= \sum_{j} {}_A\big\langle  e^{-\beta} \langle x_n \mid \langle W_{n-1}x' \mid W_{n-1} y' \rangle_A \cdot  y_n \rangle_A  \cdot u_j \mid  u_j \big\rangle\\
		&= \langle x_n \mid \langle W_{n-1}x' \mid W_{n-1} y' \rangle_A \cdot  y_n \rangle_A
		= \langle W_{n-1}x' \ox x_n \mid W_{n-1} y' \ox y_n\rangle_A.
	\end{align*}
	An inductive argument now shows that $\langle W_nx \mid W_ny\rangle_A = \langle x \mid y\rangle_A$. 
    The preceding argument extends to spans of elementary tensors, and so $W_n$ extends to an isometric linear map from $X^{\ox_A n}$ to $ F\ox_B(F^*XF)^{\ox_A n}\ox_BF^*$. 
    The $A$-bilinearity of $W_n$ follows from the fact that $Z$ commutes with elements of $A$.  
	
	The formula for $W_n^*$ can be seen to hold by using \cref{lem:Z_facts} to check that $W_n^* W_n = \Id$ on elementary tensors. We also have
	\begin{align*}
		W_n(X^{\ox_A n})= A Z \ox_A X \ox_A AZ \ox \cdots \ox A Z \ox_A X \ox A Z = P_n(F\ox_B(F^*XF)^{\ox_A n}\ox_BF^* ),
		\end{align*}
		which is complimented in $F\ox_B(F^*XF)^{\ox_A n}\ox_BF^*$.
	
	For the final statement observe that if ${}_AF_B$ is a Morita equivalence,	then $F \ox_B F^* \cong A$. 
\end{proof}
The following is an immediate consequence of \cref{lem:fock-complement}.

\begin{corl}
\label{cor:dubbya}
Let $(\phi, {}_A X_A)$ be a \mbox{$C^*$-correspondence} and let ${}_AF_B$ be a regular bi-Hilbertian bimodule.
The universal property of direct sums yields an $A$-bilinear isometry 
\[
W \colon \Fock_X \to F \ox_B \Fock_{F^*XF} \ox_B F^*
\]
such that $Wx = W_n x$ for all $x \in X^{\ox_A n}$, and $W^* y = W_n^*y$ for all $y \in F \ox_B \Fock_{F^*XF} \ox_B F$. 
In $\End_A (F \ox_B \Fock_{F^* XF^*}\ox_B F^*)$, the strict sum $P=\sum_{n\geq 0}P_n$ defines an $A$-bilinear projection, 
and $W(\Fock_X)$ is isomorphic to the complemented $A$--$A$-correspondence $P(F^* \ox_B \Fock_{F^*XF} \ox_B F ^*)$. 
If ${}_AF_B$ is a Morita equivalence, then $\Fock_X\cong F\ox_B\Fock_{F^*XF}\ox_BF^*$.
\end{corl}

\begin{rmk}
  \label{rmk:toeplitz-compacts}
  We make the following observations before stating the next result.
  Let ${}_A F_B$ be a nondegenerate C*-correspondence.
  For a \mbox{$C^*$-correspondence} $(\phi, {}_A X_A)$, the Toeplitz algebra $\cT_X$ is an $A$-algebra,
  so \cref{lem:bimod-algs} allows us to form the $B$-algebra
  \[
      F^* \ox_A\cT_X\ox_AF \cong \End^0_{\cT_X}(F^* \ox_A \cT_X).
  \]
  Similarly, \cref{lem:bimod-algs} gives us the $A$-algebra
  \[
      F \ox_B \cT_{F^*XF} \ox_B F^* \cong \End^0_{\cT_{F^*XF}}(F \ox_A \cT_{F^*XF}).
  \]
  \cref{lem:compacts-faithful} implies that $F \ox_B \cT_{F^*XF} \ox_B F^*$ acts faithfully by adjointable operators on $F \ox_B \Fock_{F^*XF}$ with the left action satisfying
  \[
      (f \ox T_x T_y^* \ox g^*) (h \ox z) =   f \ox T_xT_y^*\langle g \mid h\rangle_B z.
  \]
  The action extends to a faithful action on $F \ox_B \Fock_{F^*XF} \ox_B F^*$ by acting trivially on $F^*$, \cite[Lemma 4.7]{KatsuraCPalg}. 
\end{rmk}

\begin{lemma}
  \label{lem:bihilb_rep}
  Let $(\phi,{}_A X_A)$ be a \mbox{$C^*$-correspondence} and let ${}_AF_B$ be regular bi-Hilbertian bimodule.
  Let $\pi$ denote the left action of $A$ on $F \ox_B \cT_{F^*XF}$, and let $(u_i)_i$ be a right frame for $F_B$. 
  Define $\psi \colon X \to \End_{\cT_{F^*XF}}^0(F \ox_B \cT_{F^*XF})$ by 
  \begin{align}
  \psi(x) 
  =\sum_{i,j} u_i \ox T_{u_i^* \ox e^{-\beta/2} \cdot x \ox u_j} \ox u_j^*.
  \label{eq:sigh}
  \end{align}
  Then $(\pi,\psi)$ is a faithful representation of $(\phi,{}_A X_A)$ in $\End_{\cT_{F^*XF}}^0(F \ox_B \cT_{F^*XF})$ which induces an injective $*$-homomorphism $\alpha \colon \cT_X \to \End_{\cT_{F^*XF}}^0(F \ox_B \cT_{F^*XF})$.
\end{lemma}

\begin{rmk}
If $\cT_{F^*XF}$ is nonunital, we can replace $(u_j\ox 1)$ with $(u_j\ox b_i)$ where $b_i=(a_{i} - a_{i-1})^{1/2}$ for some approximate unit $(a_i)\in \cT_{F^*XF}$. This is true because 
$(u_j\ox b_i)$ is then a frame for $F\ox \cT_{F^*XF}$, by \cite[Proposition 1.2]{LN04}.
\end{rmk}

\begin{proof}[Proof of \cref{lem:bihilb_rep}]
First we define $\psi$ by the formula \labelcref{eq:sigh}, so that $\psi$ takes values in the adjointable operators $\End_{\cT_{F^*XF}}(F \ox_B \cT_{F^*XF})$.
We  show that $\psi(x)^* \psi(y) = \pi(\langle x \mid y\rangle_A)$. 
We compute,
\begin{align*}
    \psi(x)^* \psi(y)
    &= \sum_{i,j,k,\ell} {u_i \ox \langle u_j \ox T_{u_j^* \ox e^{-\beta/2} \cdot x \ox u_i} \mid u_k \ox T_{u_k^* \ox e^{-\beta/2}\cdot y \ox u_\ell } \rangle_B \ox (u_{\ell} \ox 1)^*}\\
    &= \sum_{i,j,k,\ell} {u_i \ox T_{u_j^* \ox e^{-\beta/2} \cdot x \ox u_i}^* \langle u_j \mid u_k\rangle_B T_{u_k^* \ox e^{-\beta/2}\cdot y \ox u_\ell }	 \ox  (u_{\ell} \ox 1)^*}\\
    &= \sum_{i,j,\ell} {u_i \ox T_{u_j^* \ox e^{-\beta/2}\cdot  x \ox u_i}^* T_{u_j^* \ox e^{-\beta/2}\cdot y \ox u_\ell }	 \ox  (u_{\ell} \ox 1)^*}\\
    &= \sum_{i,j,\ell} {u_i \ox \langle u_j^* \ox e^{-\beta/2}\cdot  x \ox u_i \mid u_j^* \ox e^{-\beta/2}\cdot y \ox u_\ell \rangle_B \ox ( u_{\ell} \ox 1)^*}.
\end{align*}
Now,
\begin{align*}
    &\sum_j	\langle u_j^* \ox e^{-\beta/2} \cdot  x \ox u_i \mid u_j^* \ox e^{-\beta/2}\cdot y \ox u_\ell\rangle_B \\
    &\quad = \sum_j \langle u_i \mid \pi(\langle e^{-\beta/2}\cdot  x \mid {}_A \langle u_j \mid u_j\rangle e^{-\beta/2} \cdot  y\rangle_A) \cdot  u_\ell \rangle_B\\
    &\quad = \langle u_i \mid \pi(\langle x \mid y\rangle_A) \cdot u_\ell\rangle_B,
\end{align*}
so in our original calculation, 
\begin{align*}
    \psi(x)^* \psi(y) &= \sum_{i,\ell} {u_i \ox \langle u_i \mid \pi(\langle x \mid y\rangle_A) u_\ell\rangle_B \ox  (u_\ell \ox 1)^*}\\
    &= \sum_{\ell} {\pi(\langle x \mid y\rangle_A)u_\ell \ox 1 \ox (u_{\ell} \ox 1)^*} 
    = \pi(\langle x \mid y\rangle_A).
\end{align*}
Since the left action on ${}_AF_B$ is faithful, it follows that $\psi$ is an isometric linear map. 
To see that $\pi(a) \psi(x) = \psi(a \cdot x)$ we compute 
\begin{align*}
    \pi(a) \sum_{i,j} {u_i \ox T_{u_i^* \ox e^{\beta/2} x \ox u_j} \ox (u_j \ox 1)^*} 
    &= \sum_{i,j,k} {u_k \langle u_k \mid \pi(a)u_i\rangle_B \ox T_{u_i^* \ox e^{\beta/2} x \ox u_j}\ox (u_j \ox 1)^*}\\
    &= \sum_{i,j,k} { u_k \ox  T_{ \langle u_k \mid \pi(a)u_i\rangle_B u_i^* \ox e^{\beta/2} x \ox u_j}\ox (u_j \ox 1)^*}\\
    &= \sum_{i,j,k} { u_k \ox  T_{ \langle u_k \mid \pi(a)u_i\rangle_B u_i^* \ox e^{\beta/2} x \ox u_j}\ox (u_j \ox 1)^*}\\
    &= \sum_{i,j,k} { u_k \ox  T_{ (u_i\langle u_i \mid \pi(a^*) u_k\rangle_B)^* \ox e^{\beta/2} x \ox u_j}\ox (u_j \ox 1)^*}\\
    &= \sum_{j,k} { u_k \ox  T_{ u_k^* \pi(a) \ox e^{\beta/2} x \ox u_j}\ox (u_j \ox 1)^*}
    = \psi(a\cdot x).
\end{align*}
A similar argument shows that $\psi(x) \pi(a) = \psi(x \cdot a)$.

With this observation, the density of $A\cdot X\cdot A$ in $X$ and the compactness of $a\Id_F$ shows that $\psi$ takes values in $\End_{\cT_{F^*XF}}^0(F \ox_B \cT_{F^*XF})$.
As $\pi$ is an injective $*$-homomorphism, the representation $(\pi,\psi)$ induces an injective $*$-homomorphism $\alpha \colon \cT_X \to \End_{\cT_{F^*XF}}^0(F \ox_B \cT_{F^*XF})$.
\end{proof}

For the next result, recall the map $W\colon \Fock_X\to F \ox_B \Fock_{F^*XF} \ox_B F^*$ and $P \in \End_A (F \ox_B \Fock_{F^* XF^*}\ox_B F^*)$ from \cref{cor:dubbya},
and the injection $\alpha \colon \cT_X \to \End_{\cT_{F^*XF}}^0(F \ox_B \cT_{F^*XF})$ from \cref{lem:bihilb_rep}.
Up to identifications, the conditional expectation defined below is simply compression by the projection $P$.

\begin{prop}
\label{lem:fock-toe}
  Let $(\phi,{}_AX_A)$ be a \mbox{$C^*$-correspondence} and let ${}_A F_B$ be a regular bi-Hilbertian bimodule.
  As in \cref{rmk:toeplitz-compacts}, identify $\End_{\cT_{F^*XF}}^0(F \ox_B \cT_{F^*XF})$ with its faithful representation as operators on $F \ox_B \Fock_{F^*XF} \ox_B F^*$. The linear map $\Phi_P \colon \End_{\cT_{F^*XF}}^0(F \ox_B \cT_{F^*XF}) \to \cT_X$ given by
  \[
    \Phi_{P}(T) = W^* P T P W 
  \]
  is a conditional expectation onto $\cT_X$ for the inclusion $\alpha$. 
  If ${}_AF_B$ is a
  Morita equivalence bimodule,
  then $\End_{\cT_{F^*XF}}^0(F \ox_B \cT_{F^*XF}) \cong \cT_X$.
\end{prop}

\begin{proof}
	Compression by $P$ defines an expectation from $\End_A(F \ox_B \Fock_{F^*XF} \ox_B F^*)$ onto $\End_A(W \Fock_X)$, which may be identified with $\End_A(\Fock_X)$ via $W$. 
    Since we can faithfully represent $\End_{\cT_{F^*XF}}^0(F \ox_B \cT_{F^*XF})$ on  $F \ox_B \Fock_{F^*XF} \ox_B F^*$, 
    we show that (up to identification) compression of an operator in $\End_{\cT_{F^*XF}}^0(F \ox_B \cT_{F^*XF})$ by $P$ yields an element of $\cT_X$.
	
	Fix $f,\,g \in F$, and $z = \bigotimes_{i=1}^{m} (f_i^* \ox x_i \ox g_i)$ and $w = \bigotimes_{i=m+1}^{m+n} (f_i^* \ox x_i \ox g_i)$ in $\Fock_{F^*XF}$. 
    Suppose that $\ell > 0$ (the $\ell = 0$ case follows from a similar, but simpler, argument) and fix $y = y_1 \ox \cdots \ox y_{\ell} \in \Fock_X$. 
    We claim that $P (f \ox T_{z} T_{w}^* \ox g^*) P Wy$ belongs to the image of $W$, from which it follows that $W^*P (f \ox T_{z} T_{w}^* \ox g^*) P W$ is an operator on $\Fock_X$. 	
    
    Fix a right frame $(u_i)$ for $F_B$ and observe that for any $f \in F$ we have $\sum_{i} \langle f\mid u_i\rangle_B \cdot  u_i^* = f^*$.    
    We compute,
	\begin{align*}
		 P (f \ox T_{z} T_{w}^* \ox g^*) PWy                                                                                       
		 &= \sum_j	P f \ox T_{z} T_{w}^*(  \langle g \mid u_j\rangle_B\cdot u_{j}^* \ox e^{-\beta/2} \cdot y_1 \ox Z \ox \cdots \ox Z \ox y_{\ell} \ox Z) \\
		 & =  P f \ox T_{z} T_{w}^* (g^* \ox e^{-\beta/2} \cdot y_1 \ox Z \ox \cdots \ox Z \ox y_{\ell} \ox Z).
	\end{align*}
	If $n > \ell$ then this is $0$, so assume $n \le \ell$.
    For the computation below, we first observe that for $f,g,h \in F$ and $x,y \in X$,
	\begin{align*}
		\sum_{j}\langle f^* \ox x \ox g \mid h^* \ox y \ox u_j\rangle_B\cdot u_j^*
		& = \sum_{j}\langle g \mid  \langle x \mid {}_A\langle f \mid h\rangle \cdot y\rangle_{A} \cdot u_j\rangle_B\cdot u_j^*\\
		&= g^* \cdot \langle x \mid {}_A\langle f \mid h\rangle \cdot y\rangle_{A}.
	\end{align*}	
   Using this at the third equality, we see that
	\begin{align*}
		 & T_{w}^*  (g^* \ox e^{-\beta/2} \cdot y_1 \ox Z \ox \cdots \ox Z \ox y_{\ell} \ox Z)                           \\
		 & \quad = \sum_{j,t} \Big\langle\bigotimes_{i=m+1}^{n+m} (f_i^* \ox x_i \ox g_i) \Bigm| g^*
		\ox e^{-\beta/2} \cdot y_1 \ox u_j \ox u_j^* \ox  e^{-\beta/2} \cdot y_2 \ox \cdots                                                                                                           \\
		 & \qquad \quad   \ox Z \ox y_{n} \ox u_{t}\Big\rangle_B \cdot u_t^*
		\ox e^{-\beta/2} \cdot y_{n+1} \ox Z \ox \cdots \ox y_{\ell} \ox Z                                                                                         \\
		 & \quad = \sum_{j,t} \
         \Big\langle\bigotimes_{i=m+2}^{n+m} (f_i^* \ox x_i \ox g_i) \Bigm| \langle f_{m+1}^* \ox x_{m+1} \ox g_{m+1} \mid g^*
		\ox e^{-\beta/2} \cdot y_1 \ox u_j\rangle_B \cdot u_j^* \ox e^{-\beta/2} \cdot y_2                                                                                                            \\
		 & \qquad \quad \ox \cdots  \ox Z \ox y_{n} \ox u_{t}\Big\rangle_B \cdot u_t^*
		\ox e^{-\beta/2} \cdot  y_{n+1} \ox Z \ox \cdots \ox y_{\ell} \ox Z                                                                                         \\
		 & \quad = \sum_{t}  \Big\langle\bigotimes_{i=m+2}^{n+m} (f_i^* \ox x_i \ox g_i) \Bigm| g_{m+1}^* \cdot \langle x_{m+1} \mid {}_A \langle f_{m+1} \mid g\rangle e^{-\beta/2} \cdot y_{1}\rangle_A \ox y_{2} \ox \cdots \\
		 & \qquad \quad   \ox Z \ox y_{n} \ox u_{t}\Big\rangle_B \cdot  u_t^*
		\ox e^{-\beta/2} \cdot y_{n+1} \ox Z \ox \cdots \ox y_{\ell} \ox Z.
	\end{align*}
	Let $a_{1} = \langle x_{m+1} \mid {}_A \langle f_{m+1} \mid g\rangle e^{-\beta/2} \cdot y_{1}\rangle_A$ and for $1 < k \le n$ inductively define
	\begin{align*}
		a_k &= \big\langle x_{m+k} \bigm| {}_A \langle f_{m+k} \mid g_{m+k-1}\rangle a_{k-1}e^{-\beta/2} \cdot y_{k}\big\rangle_A\\
		&= \big\langle e^{-\beta/2} {}_A \langle g_{m+k-1} \mid f_{m+k}\rangle \cdot x_{m+k}\bigm| a_{k-1} \cdot y_k\big\rangle_A.
	\end{align*}		
	An inductive argument shows that
	\begin{align*}
		 & T_{w}^*  (g^* \ox e^{-\beta/2} \cdot y_1 \ox Z \ox \cdots \ox Z \ox y_{\ell} \ox Z) \\
		 & \qquad =
		g_{n+m}^* \cdot a_n \ox e^{-\beta/2} \cdot y_{n+1}  \ox Z \ox \cdots \ox y_{\ell} \ox Z.
	\end{align*}
	Observe that
    $a_k = T_{ e^{-\beta/2} {}_A\langle g_{m+k-1} \mid f_{m+k} \rangle \cdot x_{m+k}}^*  a_{k-1}\cdot y_k$,
	where $T_{ e^{-\beta/2} {}_A\langle g_{m+k-1} \mid f_{m+k} \rangle \cdot x_{m+k}}^* \in \cT_X$. 
	Let 
%
\[
    \eta \coloneqq  e^{-\beta/2} {}_A \langle g \mid f_{m+1} \rangle \cdot x_{m+1} \ox  e^{-\beta/2} {}_A \langle g_{m+1} \mid f_{m+2} \rangle \cdot x_{m+2} \ox   
    \cdots  \ox  e^{-\beta/2} {}_A \langle g_{m+n-1} \mid f_{m+n} \rangle \cdot x_{m+n}.
\]
Proceeding inductively, we find that $a_n = T^*_{ \eta }  (y_1 \ox \cdots \ox y_{n})$.
	
Returning to the original calculation, we have
	\begin{align*}
		 & P (f \ox T_{z} T_{w}^* \ox g^*) P W y                                                        \\
		 & \quad = P f \ox  T_{z}
		(	g_{n+m}^* \cdot a_n \ox      e^{-\beta/2} \cdot y_{n+1}  \ox Z   \ox y_{n+2}  \ox Z \ox \cdots \ox  Z \ox y_{\ell} \ox Z )                                                                                      \\
		 & \quad = P f \ox f_1^* \ox x_1 \ox g_1^* \ox \cdots \ox f_m^* \ox x_m \ox g_m^* \ox   g_{n+m}^* \cdot  a_n \\
		 & \qquad  \quad \ox  e^{-\beta/2} \cdot y_{n+1}  \ox Z   \ox y_{n+2}  \ox Z \ox \cdots \ox  Z \ox y_{\ell} \ox Z \\
		 & \quad = P_0(f \ox f_1^*) \ox  x_1 \ox P_0 (g_1 \ox f_2^*)\ox  x_2 \ox \cdots \ox   P_0(g_{m-1} \ox f_m^*) \ox x_m \ox P_0(g_{m} \ox g_{n+m}^*) \cdot a_n                     \\
		 & \qquad \quad \ox  e^{-\beta/2} \cdot y_{n+1}  \ox Z   \ox y_{n+2}  \ox Z \ox \cdots \ox  Z \ox y_{\ell} \ox Z .
	\end{align*}
Recall that $P_0(f \ox f_1^*) = e^{-\beta/2} {}_A\langle f \mid f_1\rangle Z$ and let
	\begin{align*}
		\xi \coloneqq e^{-\beta/2} {}_A\langle f \mid f_1\rangle \cdot x_1 \ox e^{-\beta/2}  {}_A\langle g_1 \mid f_2\rangle \cdot x_2 \ox \cdots \ox e^{-\beta/2} {}_A\langle g_{m-1} \mid f_m\rangle\cdot x_m.
	\end{align*}
	Then $
		P (f \ox T_{z} T_{w}^* \ox g^*) P Wy = W(T_{\xi}  e^{-\beta}  {}_A\langle g_m \mid g_{n+m}\rangle T_{\eta}^* y),
	$
	so
	\[
	\Phi_P(f \ox T_{z} T_{w}^* \ox g^*) = T_{\xi}  e^{-\beta}  {}_A\langle g_m \mid g_{n+m}\rangle T_{\eta}^* = T_\xi P_0(g_m \ox g_{n+m}^*) T_\eta^*.
	\]
	
	With $\alpha \colon \cT_X \to \End_{\cT_{F^*XF}}^0(F \ox_B \cT_{F^*XF})$ as in \cref{lem:bihilb_rep} and $x \in X_A$, the preceding calculation shows that 
	\[
		\Phi_P(\alpha(T_x)) = W^*P \sum_{i,j} u_i \ox T_{u_i^* \ox e^{-\beta/2} x \ox u_j} \ox u_j^* P W = T_x.
	\]
	Extending the argument to elementary tensors $x,y \in \Fock_X$ we find that $\Phi_P (\alpha(T_xT_y^*)) = T_x T_y^*$
	from which it follows that $\Phi_P \circ \alpha = \Id_{\cT_X}$. 
	Tomiyama's Theorem~\cite[Theorem 1.5.10]{Brown-Ozawa} implies that $\Phi_{P}$ is a conditional expectation  

	If $F$ is a Morita equivalence, \cref{cor:dubbya} implies that $P = \Id$. 
    Consequently, $\Phi_P$ defines an isomorphism between $\End_{\cT_{F^*XF}}^0(F \ox_B \cT_{F^*XF})$ and $\cT_X$.
\end{proof}

\cref{lem:fock-toe} shows that conjugating a \mbox{$C^*$-correspondence} ${}_A X_A$ by a bi-Hilbertian bimodule ${}_A F_B$ of finite index 
yields a conditional expectation $\Phi_P \colon F \ox_B \cT_{F^*XF} \ox_B F^* \to \cT_X$. 

Now suppose that ${}_A X_A$ is a regular \mbox{$C^*$-correspondence}. If ${}_A F_B$ is regular, then so is the $B$--$B$-correspondence $F^*XF$. 
Since $\End_{B}^0(\Fock_{F^*XF})$ is a $B$-ideal in $\cT_{F^*XF}$, \cref{lem:modulequotients} implies that
\[
\frac{F \ox_B \cT_{F^*XF} \ox_B F^*}{F \ox_B \End_{B}^0(\Fock_{F^*XF}) \ox_B F^*}\cong F \ox_B \cO_{F^* X F }\ox_B F^*.
\]
In order to apply \cref{lem:quotient-postitive-map} to $\Phi_P$ we require the following lemma.

\begin{lemma}\label{lem:bihilb-projection-compacts-again}
	Let $\Phi_P \colon F \ox_B \cT_{F^*XF} \ox_B F^* \to \cT_X$ be the expectation of \cref{lem:fock-toe}. Then $\Phi_P (F \ox_B \End_{B}^0(\Fock_{F^*XF})\ox_B F^*) \subseteq \End^0_A(\Fock_X)$.
\end{lemma}
\begin{proof} 
For $f,f_1,f_2,g,g_1,g_2\in F$ and $x_1,x_2\in X$ we have
\begin{align*}
&\Phi_P(f\ox(f_1^*\ox x_1\ox g_1)\ox(f_2^*\ox x_2\ox g_2^*)^*\ox g^*)\\
&\qquad =\!e^{-\beta}{}_A\langle f\mid f_1\rangle \cdot x_1\cdot  e^{-\beta}{}_A\langle g_1\mid g_2\rangle \ox x_2^* \cdot e^{-\beta}{}_A\langle f_2\mid g\rangle,
\end{align*}	
which is a rank-1 operator on $X_A$. The result follows.
\end{proof}

%
%
%
%

\begin{thm}
\label{thm:to-be}
Let $(\phi,{}_AX_A)$ be a regular \mbox{$C^*$-correspondence} and let ${}_AF_B$ be a regular bi-Hilbertian bimodule. Then,
\begin{enumerate}
	\item the expectation $\Phi_P$ of \cref{lem:fock-toe} descends to a strict completely positive map $\wt{\Phi}_P \colon F \ox_B \cO_{F^* X F }\ox_B F^* \to \cO_X$;
	\item \cref{lem:quotient-expectation} implies that the associated KSGNS correspondence satisfies
	\[
	(\pi_{\wt{\Phi}_P}, \End_{\cO_{F^*XF}}^0(F \ox_B \cO_{F^*XF}) \ox_{\wt{\Phi}_P} \cO_X)\cong (\Id,L^2_{\cO_X}(\End_{\cO_{F^*XF}}^0(F \ox_B \cO_{F^*XF}),\wt{\Phi}_P)).
	\] 
\end{enumerate}
\end{thm}
\begin{proof}
	Together, \cref{lem:bihilb-projection-compacts-again} and \cref{lem:quotient-postitive-map} show that $\Phi_P$ descends 
    to a strict completely positive map $\wt{\Phi}_P \colon F \ox_B \cO_{F^* X F }\ox_B F^* \to \cO_X$. 
\end{proof}

Since $\cO_{F^*XF} \ox_B F^*$ induces a Morita equivalence between $\cO_{F^*XF}$ and $F \ox \cO_{F^*XF} \ox_B F^*$ we can also produce an $\cO_{F^*XF}$--$\cO_X$-correspondence.
\begin{corl}
	\label{corl:biHilb_cp_correspondence}
Let $(\phi,{}_AX_A)$ be a regular \mbox{$C^*$-correspondence} and let ${}_AF_B$ be a regular bi-Hilbertian bimodule. Then
	\begin{align*}
		& (\cO_{F^*XF} \ox_B F^*) \ox_{\wt{\Phi}_P} \cO_X \\ 
		& \qquad \coloneqq
		(\cO_{F^*XF} \ox_B F^*) \ox_{\End_{\cO_{F^*XF}}^0(F \ox_B \cO_{F^*XF})} (\End_{\cO_{F^*XF}}^0(F \ox_B \cO_{F^*XF}) \ox_{\wt{\Phi}_P} \cO_X)\\
		& \qquad \cong (\cO_{F^*XF} \ox_B F^*) \ox L^2_{\cO_X}(\End_{\cO_{F^*XF}}^0(F \ox_B \cO_{F^*XF}),\wt{\Phi}_P)
	\end{align*}
	is a nondegenerate $\cO_{F^*XF}$--$\cO_X$-correspondence. 
\end{corl}

Concretely describing  $\cO_{F^*XF}$---let alone $(\cO_{F^*XF} \ox_B F^*) \ox_{\wt{\Phi}_P} \cO_X$---in examples tends to be an involved process. We finish with an example from covering spaces.
\begin{example}

Let $M$ be a compact Hausdorff space with a homeomorphism $\gamma\colon M\to M$ and dual automorphism $\gamma^* \colon C(M)\to C(M)$. 
Then ${}_{\gamma^*} C(M)_{C(M)}$ is a Morita equivalence bimodule. For $g_1,g_2,g_3\in C(M)$ the left and right module structures are given by
\[
(g_1\cdot g_2\cdot g_3)(x)=g_1(\gamma(x))g_2(x)g_3(x),\quad x\in M.
\]
The right inner product is the obvious one, and the left inner product is
\[
{}_{C(M)}\langle g_1\mid g_2\rangle(x)=(\overline{g_1}g_2)(\gamma^{-1}(x)),\quad x\in M.
\]

Suppose that $\pi\colon \wt{M} \to M$ is a finite-to-one covering map and let $F_{C(\wt{M})} \coloneqq C(\wt{M})_{C(\wt{M})}$. 
Then ${}_{\pi^*}F_{C(\wt{M})}$ is a bi-Hilbertian bimodule with left $C(M)$-valued inner product given by
\[
{}_{C(M)}\langle f_1 \mid f_2\rangle(x) = \sum_{\wt{x} \in \pi^{-1}(x)} f_1(\wt{x})\, \overline{f_2(\wt{x})} \qquad f_1,f_2 \in C(\wt{M}). 
\]
We can form the bi-Hilbertian $C(\wt{M})$-bimodule $F^*\ox_{\gamma^*}C(M)_{}\ox_{C(M)} F$ and consider the resulting noncommutative dynamics on $C(\widetilde{M})$.
An energetic exercise shows that
\[
F^*\ox_{\gamma^*}C(M)_{}\ox_{C(M)}F\cong C\big(\wt{M}{}_{\pi}\times_{\gamma\circ \pi}\wt{M}\big)=C(\{(\wt{x},\wt{y}) \in \wt{M} \times \wt{M} \colon \pi(\wt{x})=\gamma\circ\pi(\wt{y})\})
\]
as bi-Hilbertian $C(\wt{M})$-bimodules. 
The right inner product on $F^*\ox_{\gamma^*}C(M)_{}\ox_{C(M)} F$ is given by
\begin{align}
\langle f^*_1\ox h\ox f_2\mid g_1^*\ox k\ox g_2\rangle_{C(\wt{M})}(\tilde{x})
&=\Big\langle f_2\Bigm| \pi^*\big(\langle h\mid \gamma^*\langle f_1^*\mid g_1^*\rangle_{C(M)}k\rangle_{C(M)}\big)g_2\Big\rangle_{C(\wt{M})}(\wt{x})\nonumber\\
&=\ol{f_2(\wt{x})}\big(\langle h\mid \gamma^*\langle f_1^*\mid g_1^*\rangle_{C(M)}k\rangle_{C(M)}\big)(\pi(\wt{x}))g_2(\wt{x})\nonumber\\
&=\ol{f_2(\wt{x})}\ol{h(\pi(\wt{x}))}\sum_{\tilde{y}\in\pi^{-1}(\gamma(\pi(\wt{x})))}(f_1 \ol{g_1})(\tilde{y})k(\pi(\wt{x}))g_2(\wt{x}).
\label{eq:eff-exx-eff}
\end{align}

We use the notion of topological graphs from \cite{KatsuraTopGraph}.
Regarding ${}_{\gamma^*} C(M)_{C(M)}$ as   the graph correspondence of the topological graph $M \overset{\gamma}{\leftarrow} M \overset{\Id}{\rightarrow} M$, the correspondence $F^*\ox_{\gamma^*}C(M)_{}\ox_{C(M)}F$ is isomorphic to the graph correspondence of the fibre product $\wt{M} \overset{p_1}{\leftarrow} \wt{M}{}_{\pi}\times_{\gamma\circ \pi}\wt{M} \overset{p_2}{\rightarrow} \wt{M}$, where $p_1$ and $p_2$ are the projections onto the first and second components of $\wt{M}{}_{\pi}\times_{\gamma\circ \pi}\wt{M}$.
The correspondence between Cuntz--Pimsner algebras that we obtain from \cref{corl:biHilb_cp_correspondence} is
\[
\cO_{F^*C(M)F}\ox F^*\ox_{\tilde{\Phi}_P}\cO_{X}.
\]
Treatment of the algebra $\cO_{F^*C(M)F}$ is difficult, even when the dynamics on $M$ is trivial, and we leave a detailed description of this algebra and the resulting correspondence to another place.
\end{example}


\end{document}